\documentclass[11pt]{amsart}

\usepackage{amsmath,amsthm,amsfonts,amssymb}
\usepackage{graphicx}

\usepackage{upref}

\let\<\langle
\let\>\rangle

\let\uml\"

\usepackage{tikz}
\usepackage{tikz}\usetikzlibrary{decorations.markings,backgrounds,calc}
\tikzset{middlearrow/.style={
        decoration={markings,
            mark= at position 0.5 with {\arrow[scale=1.25]{#1}} ,
        },
        postaction={decorate}
    }
}

\newcommand{\mcC}{\mathcal{C}}

\newcommand{\mcZ}{\mathcal{Z}}
\newcommand{\mfA}{\mathfrak{A}}
\newcommand{\mfS}{\mathfrak{S}}
\newcommand{\mbbC}{\mathbb{C}}
\newcommand{\mbbI}{\mathbb{I}}
\newcommand{\mbbN}{\mathbb{N}}
\newcommand{\mbbQ}{\mathbb{Q}}
\newcommand{\mbbR}{\mathbb{R}}
\newcommand{\mbbZ}{\mathbb{Z}}
\newcommand{\e}{\epsilon}
\newcommand{\g}{\gamma}

\newcommand{\n}{\nu}

\newcommand{\ps}{\psi}
\newcommand{\s}{\sigma}
\newcommand{\z}{\zeta}

\renewcommand{\(}{\left(}
\renewcommand{\[}{\left[}
\newcommand{\lcb}{\left\{}
\renewcommand{\)}{\right)}
\renewcommand{\]}{\right]}
\newcommand{\rcb}{\right\}}

\renewcommand{\bar}[1]{\overline{#1}}

\newcommand{\abs}[1]{\left|#1\right|}

\DeclareMathOperator{\FPdim}{FPdim}
\DeclareMathOperator{\End}{End}
\DeclareMathOperator{\Hom}{Hom}
\DeclareMathOperator{\Irr}{Irr}
\DeclareMathOperator*{\SL}{SL}

\DeclareMathOperator{\Fib}{Fib}
\DeclareMathOperator{\Rep}{Rep}
\DeclareMathOperator{\sVec}{sVec}
\DeclareMathOperator{\ord}{ord}
\DeclareMathOperator{\Tr}{Tr}

\numberwithin{equation}{section}
\newtheorem{theorem}{Theorem}[section]

\newtheorem*{thm*}{Theorem}

\newtheorem{corollary}[theorem]{Corollary}

\newtheorem{conjecture}[theorem]{Conjecture}
\newtheorem{prop}[theorem]{Proposition}
\newtheorem{proposition}[theorem]{Proposition}
\theoremstyle{definition}
\newtheorem{remark}[theorem]{Remark}

\newcommand{\thmref}[1]{Theorem $\ref{#1}$}
\newcommand{\conjref}[1]{Conjecture $\ref{#1}$}

\newcommand{\propref}[1]{Proposition $\ref{#1}$}

\newcommand{\corref}[1]{Corollary $\ref{#1}$}

\makeatletter
  \newcommand{\addresseshere}{
  \enddoc@text\let\enddoc@text\relax
}
\makeatother

\renewcommand{\th}{\theta}
\renewcommand{\d}{\delta}
\renewcommand{\a}{\alpha}
\renewcommand{\b}{\beta}
\renewcommand{\l}{\lambda}

\renewcommand{\t}{\tau}

\begin{document}
\title[]{Rank 4 premodular categories}
\date{\today}
\author{Paul Bruillard\textsuperscript{$\dagger$}}
\email{Paul.Bruillard@pnnl.gov}
\dedicatory{
\ \\
\textsuperscript{$\dagger$}Pacific Northwest National Lab, 902 Battelle Boulevard,
Richland, Washington, U.S.A.\\
%
%
%
%
%
%
}
\thanks{\textit{PNNL Information Release:} PNNL-SA-111549}
\thanks{A portion of this paper was produced at a workshop at the American Institute of
Mathematics, whose support and hospitality are gratefully acknowledged.
The research described in this paper was, in part, conducted under the
Laboratory Directed Research and Development Program at PNNL, a multi-program
national laboratory operated by Battelle for the U.S. Department of Energy.  }

\maketitle

\begin{abstract} 
  We consider the classification problem for rank 4 premodular categories. We
  uncover a formula for the $2^{\text{nd}}$ Frobenius-Schur indicator of a
  premodular category, and complete the classification of rank 4 premodular
  categories (up to Grothendieck equivalence). 
\end{abstract}

\section{Introduction}
  \label{Section: Introduction}
  The theory of fusion categories is a natural generalization of representation
  theory-- not only of finite groups, but of Lie groups and Hopf algebras and
  so, in some sense, their classification began with the classification of
  groups and their representations. At the time of this writing, a complete
  classification has only been completed for rank 2 and 3 fusion categories
  \cite{O2,O6}. While the classification problem for fusion categories is largely believed to
  be intractable, several natural structures can be imposed on fusion categories
  to make them more amenable to study.

  One such structure is that of braiding. This gives rise to a kind of
  commutativity and indeed forces the underlying Grothendieck semiring to be
  commutative. On the other hand, one might expect that the two natural notions
  of dimension in the theory coincide, leading to pseudo-unitary fusion
  categories. If study is restricted to pseudo-unitary fusion categories, then
  it is known that the category is also spherical \cite{ENO1}. The appearance of
  a spherical structure is perhaps not surprising as there are no known examples
  of non-spherical fusion categories at this time.

  Even with the addition of these structures, a full classification is believed
  to be out of reach as it would include a classification of finite groups.
  However, these categories admit a stratification by degeneracy of the
  $S$--matrix into symmetric, properly premodular, and modular categories. The
  representation categories fall naturally in the symmetric case and in fact
  completely fill it out \cite{D1}. At the other end of the spectrum, a
  large amount of work has gone into understanding modular categories spurred by
  their relationship to rational conformal field theories, quantum computation,
  link invariants, and 3-manifold invariants \cite{W1}\cite{Tu1}\cite{BKi}.
  However, recently premodular categories have been shown to provide the
  algebraic underpinnings of $\(3+1\)$-dimensional topological quantum field
  theories and thereby govern topological insulators and some high--$T_{c}$
  superconductors \cite{WW1}. In addition to their innate uses, premodular
  categories give rise to modular categories through the double construction.

  Classification of premodular categories has been completed for rank 2 and 3
  \cite{O2}\cite{O4} and in this paper we extend the classification to rank 4.
  Since the techniques commonly applied in the modular setting do not apply in
  the premodular setting new tools are developed. Specifically, 
  the following formula for the $2^{\text{nd}}$ Frobenius-Schur indicator for a
  self-dual object is determined in terms of the premodular datum. 
  \begin{align*}
    \n_{2}\(X_{a}\)&=\frac{1}{D^{2}}\sum_{b,c}N_{bc}^{a}d_{b}d_{c}\(\frac{\th_{b}}{\th_{c}}\)^{2}-\th_{a}\sum_{\g\in\mcC'\setminus\mbbI}d_{\g}\Tr\(R_{\g}^{aa}\).
  \end{align*}

  We will begin by reviewing the theory of modular and premodular categories.
  Having dispensed with these preliminaries, a formula for the $2^{\text{nd}}$
  Frobenius-Schur indicator will be derived in the premodular setting. As an
  application of this indicator, the rank 4 premodular categories will then be
  classified. In conjunction with \cite{RSW}, this will complete the
  classification of rank 4 premodular and modular categories.
%

\section{Preliminaries}
  \label{Section: Preliminaries}
  A premodular category $\mcC$ is a braided, balanced, and fusion category.
  Furthermore, if the $S$-matrix is invertible then $\mcC$ is said to be
  modular. Every premodular category $\mcC$ is a ribbon category and as such
  enjoys a graphical calculus. A brief account of this calculus in addition to
  some salient algebraic relations will be given and further detail can be found
  in \cite{BKi}\cite{K1}\cite{Tu1}.

  \subsection{Pivotal structure and dimensions}
    By virtue of being a fusion category, $\mcC$ is semisimple and we will denote
    the isomorphism classes of the simple objects by $\mbbI=X_{0},\ldots, X_{n-1}$
    where $n$ is known as the \textit{rank} of $\mcC$. Furthermore, $\mcC$ is
    balanced and hence pivotal. This structure manifests itself through a duality
    $*$ acting by $X_{a}^{*}=X_{a^{*}}$. Such a duality induces an involution on
    the labeling set for the simple objects and can be encoded by the
    \textit{charge conjugation matrix} $C_{ab}=\d_{ab^{*}}$. Graphically, a
    nontrivial simple object $X_{a}$ is denoted by an upward arrow and its dual by
    a downward arrow,
    \begin{align}
      \vcenter{\hbox{
      \begin{tikzpicture}[xscale=0.23,yscale=0.23]
        \path[draw=black, middlearrow={stealth reversed}] (0,2)--(0,-2);
        \node at (0,-3) {\scalebox{0.75}{$a$}};
      \end{tikzpicture}
      }}
      \quad
      \vcenter{\hbox{
      \begin{tikzpicture}[xscale=0.23,yscale=0.23]
        \path[draw=black, middlearrow={stealth}] (0,2) -- (0,-2);
        \node at (0,-3) {\scalebox{0.75}{$a$}};
      \end{tikzpicture}
      }}
    \end{align}

    For the trivial object, $X_{0}=\mbbI$, no arrow is drawn. Note that for a
    self-dual object the arrow may be safely omitted. The pivotal structure of
    $\mcC$ further provides a collection of evaluation and co-evaluation maps
    \begin{equation}
      \label{Evaluation and Coevaluation}
      \begin{split}
        &ev_{X}:X^{*}\otimes X\to \mbbI\\
        &coev_{X}:\mbbI\to X\otimes X^{*}
      \end{split}
    \end{equation}

    These maps are given by the cup and cap
    \begin{equation}
      \label{Graphical Evaluation and Coevaluation}
      coev=\vcenter{\hbox{
      \begin{tikzpicture}[xscale=0.23,yscale=0.23]
        \path[draw=black, middlearrow={stealth reversed}] (0,0) arc (180:360:1 and 1);
      \end{tikzpicture}
      }}
      \quad
      ev=\vcenter{\hbox{
      \begin{tikzpicture}[xscale=0.23,yscale=0.23]
        \path[draw=black, middlearrow={stealth}] (0,0) arc (0:180:1 and 1);
      \end{tikzpicture}
      }}
    \end{equation}

    Compatibility of such maps give rise to the allowed graphical moves:
    \begin{equation}
      \label{Evaluation Coevaluation moves}
      \vcenter{\hbox{
      \begin{tikzpicture}[xscale=0.23,yscale=0.23]
        \path[draw=black] (0,0) arc (180:360:1 and 1);
        \path[draw=black] (4,0) arc (0:180:1 and 1);
        \path[draw=black, middlearrow={stealth}] (0,0) -- (0,2);
        \path[draw=black, middlearrow={stealth reversed}] (4,0) -- (4,-2);
      \end{tikzpicture}
      }}=\vcenter{\hbox{
      \begin{tikzpicture}[xscale=0.23,yscale=0.23]
        \path[draw=black] (2,0) arc (0:180:1 and 1);
        \path[draw=black] (2,0) arc (180:360:1 and 1);
        \path[draw=black, middlearrow={stealth reversed}] (0,0) -- (0,-2);
        \path[draw=black, middlearrow={stealth}] (4,0) -- (4,2);
      \end{tikzpicture}
      }}=\vcenter{\hbox{
      \begin{tikzpicture}[xscale=0.23,yscale=0.23]
        \path[draw=black, middlearrow={stealth reversed}] (0,2)--(0,-2);
      \end{tikzpicture}
      }}
    \end{equation}

    A pivotal category also comes equipped with a family of natural isomorphisms
    $j_{X}:X\to X^{**}$. The presence of these maps give rise to two canonical
    traces called left and right pivotal traces \cite{NS2}. In a spherical
    category, these traces coincide and so, for $f\in\End_{\mcC}\(X\)$, one simply
    writes $\Tr_{\mcC}\(f\)$. By the coherence theorems, it is known that every
    premodular category is equivalent to a strict premodular category and so we
    will, without loss of generality, restrict our attention to strict categories.
    One benefit of focusing on strict categories is that the isomorphisms $j_{X}$
    can be removed, which greatly simplifies the graphical calculus. For instance,
    taking the trace of $id_{X_{a}}$ allows one to define the \textit{dimension} of
    $X_{a}$ and the \textit{global dimension}, $D^{2}$. These dimensions are
    graphically given by
    \begin{align}
    \label{GraphicalDimensions}
    \dim\(X_{a}\)=d_{a}=
    \vcenter{\hbox{
    \begin{tikzpicture}[xscale=0.23,yscale=0.23]
      \path[draw=black, middlearrow={stealth reversed}] (0,0) arc (0:360:1 and 1);
      \node at (-2.5,0.75) {\scalebox{0.75}{$a$}};
    \end{tikzpicture}
    }},
    \quad
    D^{2}=\dim\(\mcC\)=
    \vcenter{\hbox{
    \begin{tikzpicture}[xscale=0.23,yscale=0.23]
      \path[draw=black] (0,0) arc (0:360:1 and 1);
    \end{tikzpicture}
    }}
    :=\displaystyle{\sum_{b\in\Irr\(\mcC\)}}d_{b}
    \vcenter{\hbox{
    \begin{tikzpicture}[xscale=0.23,yscale=0.23]
      \path[draw=black, middlearrow={stealth reversed}] (0,0) arc (0:360:1 and 1);
      \node at (-2.5,0.75) {\scalebox{0.75}{$b$}};
    \end{tikzpicture}
    }}
    \end{align}

  \subsection{Fusion and splitting spaces}
    $\mbbC$-linearity of $\mcC$ endows $\Hom_{\mcC}\(V,W\)$ with the structure of
    a complex vector space for all objects $V$ and $W$ in $\mcC$. However, certain
    families of Hom-spaces are distinguished due to semisimplicity, they are the
    \textit{fusion spaces} $V_{ab}^{c}=\Hom_{\mcC}\(X_{a}\otimes X_{b},X_{c}\)$
    and the \textit{splitting spaces} $V^{ab}_{c}=\Hom_{\mcC}\(X_{c},X_{a}\otimes
    X_{b}\)$. In the course of this work a basis of the splitting space will be
    denoted by$\lcb \ps_{c,i}^{ab}\rcb$ and the dual basis of the fusion space is
    given by $\lcb \ps_{ab,j}^{c}=\(\ps_{c,j}^{ab}\)^{\dagger}\rcb$. These bases
    are graphically depicted by
    \begin{equation}
      \label{Graphical Fusion/Splitting Bases}
      \vcenter{\hbox{\begin{tikzpicture}[xscale=0.23,y=6.6]
        \path[draw=black] (0,0) arc (180:360:1 and 1);
        \path[draw=black, middlearrow={stealth}] (0,0)--(0,1.5);
        \path[draw=black, middlearrow={stealth}] (2,0)--(2,1.5);
        \path[draw=black, middlearrow={stealth reversed}] (1,-1)--(1,-3.5);
        \node at (-0.5,1) {\scalebox{0.75}{$a$}};
        \node at (2.5,1) {\scalebox{0.75}{$b$}};
        \node at (1.5,-3.5) {\scalebox{0.75}{$c$}};
        \node[draw=black, fill=white, shape=circle, inner sep=0.5pt, minimum width=10pt] at (1,-1) {\scalebox{0.75}{$i$}};
      \end{tikzpicture}
      }
      }
      \quad\text{ and }\quad
      \vcenter{\hbox{
      \begin{tikzpicture}[xscale=0.23,yscale=0.23]
        \path[draw=black] (0,0) arc (-180:-360:1 and 1);
        \path[draw=black, middlearrow={stealth reversed}] (0,0)--(0,-1.5);
        \path[draw=black, middlearrow={stealth reversed}] (2,0)--(2,-1.5);
        \path[draw=black, middlearrow={stealth}] (1,1)--(1,3.5);
        \node at (-0.5,-1) {\scalebox{0.75}{$a$}};
        \node at (2.5,-1) {\scalebox{0.75}{$b$}};
        \node at (1.5,3.25) {\scalebox{0.75}{$c$}};
        \node[draw=black, fill=white, shape=circle, inner sep=0.5pt, minimum width=10pt] at (1,1) {\scalebox{0.75}{$j$}};
      \end{tikzpicture}
      }
      }
    \end{equation}

    respectively. The normalization of these bases will always be such that
    \begin{equation}
      \label{Fusion/Splitting Normalization}
      \th\(a,b,c\)\d_{ij}=
        \vcenter{\hbox{\begin{tikzpicture}[xscale=0.23,yscale=0.23]
        \path[draw=black] (0,0) arc (180:360:1 and 1);
        \path[draw=black] (0,0) -- (0,1);
        \path[draw=black] (2,0) -- (2,1);
        \path[draw=black] (4,1) arc (0:180:1 and 1);
        \path[draw=black] (0,1) arc (-180:-360:3 and 3);
        \path[draw=black,middlearrow={stealth reversed}] (1,-1) -- (1,-4);
        \path[draw=black] (0,-5) arc (-180:-360:1 and 1);
        \path[draw=black] (0,-5) -- (0,-6);
        \path[draw=black] (2,-5) -- (2,-6);
        \path[draw=black] (2,-6) arc (180:360:1 and 1);
        \path[draw=black] (0,-6) arc (180:360:3 and 3);
        \path[draw=black,middlearrow={stealth}] (4,1) -- (4,-6);
        \path[draw=black,middlearrow={stealth}] (6,1) -- (6,-6);
        \node at (1.5,-2) {\scalebox{0.75}{$c$}};
        \node at (-0.5, 0) {\scalebox{0.75}{$a$}};
        \node at (2.5,0) {\scalebox{0.75}{$b$}};
        \node[draw=black, fill=white, shape=circle, inner sep=0.5pt, minimum width=10pt] at (1,-1) {\scalebox{0.75}{$i$}};
        \node[draw=black, fill=white, shape=circle, inner sep=0.5pt, minimum width=10pt] at (1,-4) {\scalebox{0.75}{$j$}};
      \end{tikzpicture}
      }
      }
    \end{equation}

    where $\th\(a,b,c\)=\sqrt{d_{a}d_{b}d_{c}}$ is the \textit{theta symbol}.
    Further note that this normalization is consistent with the graphical
    dimensions given in equation \eqref{GraphicalDimensions}, \textit{i.e.}, $b=a^{*}$ and
    $c=0$.  This particular symbol appears in the decomposition of $id_{X_{a}\otimes X_{a}}$ as
    \begin{equation}
      \label{SplitThenFuse}
      \vcenter{\hbox{
      \begin{tikzpicture}[xscale=0.23,yscale=0.23]
        \path[draw=black, middlearrow={stealth}] (0,-6)--(0,1);
        \path[draw=black, middlearrow={stealth}] (2,-6)--(2,1);
        \node at (-0.5,-3.5) {\scalebox{0.75}{$a$}};
        \node at (2.5,-3.5) {\scalebox{0.75}{$b$}};
      \end{tikzpicture}
      }}=\sum_{c\in \Irr\mcC}\sum_{i\in V_{c}^{ab}}\sum_{j\in V_{ab}^{c}}\frac{d_{c}}{\th\(a,b,c\)}
      \vcenter{\hbox{
      \begin{tikzpicture}[xscale=0.23,yscale=0.23]
        \path[draw=black] (0,0) arc (180:360:1 and 1);
        \path[draw=black, middlearrow={stealth}] (0,0)--(0,1);
        \path[draw=black, middlearrow={stealth}] (2,0)--(2,1);
        \path[draw=black, middlearrow={stealth reversed}] (1,-1)--(1,-4);
        \path[draw=black] (2,-5) arc (0:180:1 and 1);
        \path[draw=black, middlearrow={stealth reversed}] (2,-5) -- (2,-6);
        \path[draw=black, middlearrow={stealth reversed}] (0,-5)--(0,-6);
        \node at (-0.5,0.5) {\scalebox{0.75}{$a$}};
        \node at (2.5,0.5) {\scalebox{0.75}{$b$}};
        \node[draw=black, fill=white, shape=circle, inner sep=0.5pt, minimum width=10pt] at (1,-1) {\scalebox{0.75}{$i$}};
        \node[draw=black, fill=white, shape=circle, inner sep=0.5pt, minimum width=10pt] at (1,-4) {\scalebox{0.75}{$j$}};
        \node at (-0.5,-6) {\scalebox{0.75}{$a$}};
        \node at (2.5,-6) {\scalebox{0.75}{$b$}};
        \node at (1.5,-2) {\scalebox{0.75}{$c$}};
      \end{tikzpicture}
      }}
    \end{equation}

    The dimension of the fusion space $\Hom_{\mcC}\(X_{a}\otimes X_{b},X_{c}\)$,
    $N_{ab}^{c}$, gives the multiplicity of $X_{c}$ appearing in $X_{a}\otimes
    X_{b}$, and is called a \textit{fusion coefficient}. The fusion coefficients
    are generally collected into \textit{fusion matrices}
    $\(N_{a}\)_{bc}=N_{ab}^{c}$ and furnish a representation of the Grothendieck
    semiring $Gr\(\mcC\)$ \cite{HR1}. Since the fusion coefficients are nonnegative
    integers, the Frobenius-Perron Theorem can be applied to deduce the existence
    of a largest eigenvalue of $N_{a}$, such an eigenvalue is called the
    \textit{Frobenius-Perron dimension} or \textit{FP-dimension} of $X_{a}$ and is
    denoted $\FPdim\(X_{a}\)$. One says that a premodular category is
    \textit{pseudo-unitary} if $\FPdim\(X_{a}\)=\dim\(X_{a}\)$ for all $a$. The
    \textit{global FP-dimension} of the category is defined by
    $\FPdim\(\mcC\)=\sum_{a}\FPdim\(X_{a}\)^{2}$. If the global FP-dimension
    is an integer, the category is said to be \textit{weakly integral} and if
    $\FPdim\(X_{a}\)\in\mbbZ$ for all $a$ then one says $\mcC$ is
    \textit{integral}. Finally, duality and braiding endow the fusion matrices
    with the following symmetries \cite{BKi}:
    \begin{equation}
      \label{FusionSymmetries}
      \begin{split}
        N_{ab}^{c}=N_{ba}^{c}&=N_{ac^{*}}^{b^{*}}=N_{a^{*}b^{*}}^{c^{*}}\\
        N_{ab^{*}}^{0}=1,\quad N_{a^{*}}&=N_{a}^{T},\quad N_{a}N_{b}=N_{b}N_{a}.
      \end{split}
    \end{equation}

  \subsection{Spherical structure}
    The braiding and spherical structure give rise to canonical elements $\th_{a}\in
    \End_{\mcC}\(X_{a}\)$ called \textit{twists}. Since $\End_{\mcC}\(X_{a}\)$ is
    one dimensional, the twists are scalar multiples of the identity, also denoted
    $\th_{a}$. Graphically, we have
    \begin{align*}
      \th_{a}
      \vcenter{\hbox{
      \begin{tikzpicture}[xscale=0.23,yscale=0.23]
        \path[draw=black, middlearrow={stealth reversed}] (0,1.5) -- (0,-2.5);
        \node at (0,-3) {\scalebox{0.75}{$a$}};
      \end{tikzpicture}
      }}=
      \vcenter{\hbox{
      \begin{tikzpicture}[xscale=0.23,yscale=0.23]
        \path[draw=black, middlearrow={stealth}] (-1,-2.5) -- (-1,-1);
        \path[draw=black] (0,0) arc (90:180:1 and 1);
        \path[draw=black] (0,-1) arc (-90:90:0.5 and 0.5);
        \path[draw=black] (0,-1) arc (-90:-130:1 and 1);
        \path[draw=black] (-1,0) arc (180:200:1 and 1);
        \path[draw=black] (-1,0) -- (-1,1.5);
        \node at (-1,-3) {\scalebox{0.75}{$a$}};
      \end{tikzpicture}
      }}
    \end{align*}

    The celebrated Vafa Theorem tells us that these twists are roots of
    unity\cite{Va1}. For convenience, the twists are collected into the diagonal
    matrix $T_{ab}=\d_{ab}\th_{b}$ called the \textit{T-matrix}.

  \subsection{Braiding}
    The braiding in $\mcC$ is given by elements\ \\
    $R_{ab}\in\Hom_{\mcC}\(X_{a}\otimes X_{b},X_{b}\otimes X_{a}\)$. Coupling
    these maps with the splitting spaces, one can define the $R$-matrices
    $\(R_{c}\)_{ab}=R^{ab}_{c}$, where $R_{c}^{ab}$ is obtained by ``braiding
    $X_{a}$ with $X_{b}$ in the $X_{c}$ channel.'' In fact, the bases of the
    splitting space $V_{c}^{ab}$ can be chosen to diagonalize $R_{c}^{ab}$ by
    $R_{c}^{ab}\ps_{c,i}^{ab}=R_{c,i}^{ab}\ps_{c,i}^{ab}$ \cite{K1}.
    Pictorially, this is given by
    \begin{align*}
      R_{c,i}^{ab}
      \vcenter{\hbox{
      \begin{tikzpicture}[xscale=0.23,yscale=0.23]
        \path[draw=black] (0,0) arc (180:360:1 and 1);
        \path[draw=black, middlearrow={stealth}] (0,0)--(0,4);
        \path[draw=black, middlearrow={stealth}] (2,0)--(2,4);
        \path[draw=black, middlearrow={stealth reversed}] (1,-1)--(1,-4);
        \node at (-0.5,3.5) {\scalebox{0.75}{$a$}};
        \node at (2.5,3.5) {\scalebox{0.75}{$b$}};
        \node at (1.5,-3.5) {\scalebox{0.75}{$c$}};
        \node[draw=black, fill=white, shape=circle, inner sep=0.5pt, minimum width=10pt] at (1,-1) {\scalebox{0.75}{$i$}};
      \end{tikzpicture}
      }}
      =
      \vcenter{\hbox{
      \begin{tikzpicture}[xscale=0.23,yscale=0.23]
        \path[draw=black] (-2,0) arc (180:360:1 and 1);
        \path[draw=black] (-2,0)--(-2,1);
        \path[draw=black] (0,0)--(0,1);
        \path[draw=black, middlearrow={stealth reversed}] (-1,-1)--(-1,-4);
        \path[draw=black] (-2,1) arc (-180:-240:2 and 1.25);
        \path[draw=black] (-1,2.08253) arc (300:360:2 and 1.25);
        \path[draw=black] (0,1) arc (0:50:2 and 1.25);
        \path[draw=black] (-2,3.16506) arc (180:230:2 and 1.25);
        \path[draw=black, middlearrow={stealth}] (-2,3.16506) -- (-2,4);
        \path[draw=black, middlearrow={stealth}] (0,3.16506) -- (0,4);
        \node at (-2.5,3.5) {\scalebox{0.75}{$a$}};
        \node at (0.5,3.5) {\scalebox{0.75}{$b$}};
        \node at (-0.5,-3.5) {\scalebox{0.75}{$c$}};
        \node[draw=black, fill=white, shape=circle, inner sep=0.5pt, minimum width=10pt] at (-1,-1) {\scalebox{0.75}{$i$}};
      \end{tikzpicture}
      }}
    \end{align*}

    These braidings give rise to a family of natural isomorphisms
    $c_{ab}=R_{ba}R_{ab}$ in $\End_{\mcC}\(X_{a}\otimes X_{b}\)$ which can be
    traced to define the \textit{S}-matrix 
    \begin{equation}
      \label{SMatrixDefinition}
      \tilde{s}_{ab}=\Tr_{\mcC}\(c_{a\bar{b}}\)=
      \vcenter{\hbox{
      \begin{tikzpicture}[xscale=0.23,yscale=0.23]
        \path[draw=black, middlearrow={stealth}] ({2*cos(50)},{2*sin(50)}) arc (50:390:2 and 2);
        \path[draw=black, middlearrow={stealth}] ({2*cos(235)+3},{2*sin(235)}) arc (235:575:2 and 2);
        \node at (0.5,0) {\scalebox{0.75}{$\bar{b}$}};
        \node at (2.5,0) {\scalebox{0.75}{$a$}};
      \end{tikzpicture}}}
    \end{equation}

  \subsection{Algebraic identities}
    The $S$-matrix is highly symmetric and, in fact we have
    \begin{align}
      \label{SMatrixSymmetries}
      \tilde{s}_{a^{*}b}^{*}=\tilde{s}_{ab}=\tilde{s}_{ba}=\tilde{s}_{a^{*}b^{*}},\quad \tilde{s}_{a0}=d_{a}.
    \end{align}

    In the course of this work the tuple
    $\(\tilde{S},\;T,\;N_{0},\;\ldots,\;N_{n}\)$ will be referred to as premodular
    datum. Perhaps not surprisingly, the matrices comprising premodular datum are
    strongly related. For instance, an elementary application of the graphical
    calculus leads to the \textit{balancing relation} \cite{BKi}
    \begin{equation}
      \label{Balancing}
      \tilde{s}_{ab}=\th_{a}^{-1}\th_{b}^{-1}\displaystyle{\sum_{c}}N_{a^{*}b}^{c}\th_{c}d_{c}.
    \end{equation}

    Additionally, one can show that the columns of $S$--matrix are eigenvectors of
    the fusion matrices. In a modular category, this leads to the well-known
    Verlinde Formula, while in the premodular setting it is shown in \cite{M4}
    that
    \begin{equation}
      \label{PreVerlinde}
      \tilde{s}_{ab}\tilde{s}_{ac}=d_{a}\displaystyle{\sum_{\ell}}N_{bc}^{\ell}\tilde{s}_{a\ell}.
    \end{equation}

    It can further be shown that the $S$-- and $T$--matrices are related by 
    \begin{equation}
      \label{STCubed and STInvCubed}
      \begin{split}
        \(\tilde{S}T\)^{3}&=p^{+}\tilde{S}^{2},\\
        \(\tilde{S}T^{-1}\)^{3}&=p^{-}\tilde{S}^{2}C.
      \end{split}
    \end{equation}

    where $C_{a,b}=\d_{a,b^{*}}$ is the \textit{charge conjugation} matrix, and $p^{\pm}$ are the \textit{Gauss sums}:
    \begin{equation}
      p^{\pm}=\displaystyle{\sum_{a}}\th_{a}^{\pm}d_{a}^{2}.
    \end{equation}

    If $\det\(\tilde{S}\)\neq0$ then $\mcC$ is said to be modular and the
    additional identities
    \begin{equation}
      \label{SSdagger}
      \tilde{S}\tilde{S}^{\dagger}=D^{2}\mbbI\quad\text{ and }\quad p^{+}p^{-}=D^{2},
    \end{equation}

    are acquired, from which it is clear that $\tilde{S}$ and $T$ furnish a
    projective representation of the modular group $\SL\(2,\mbbZ\)$.

    $\mcC$ is said to be \textit{symmetric} if $\tilde{s}_{ab}=d_{a}d_{b}$ for all
    $a$ and $b$. One can view symmetric categories as completely degenerate
    premodular categories while modular categories are completely nondegenerate.
    It is between these two extremes that we will be focusing our attention and so
    we define a \textit{properly premodular category} $\mcC$ to be a premodular
    category that is neither symmetric nor modular. In this way, symmetric,
    properly premodular, and modular categories partition the class of premodular
    categories.

  \subsection{The M\"{u}ger center and finiteness}
    The braiding can be used to define the \textit{M\"{u}ger center} of a
    premodular category by \cite{M4}
    \begin{align}
      \label{MugerCenterDefinition}
      \mcC'=\lcb X\in\mcC\mid c_{X,Y}=id_{X\otimes Y},\;\forall Y\in\mcC\rcb.
    \end{align}

    The elements of the center are often called \textit{central} or
    \textit{transparent} \cite{M4}\cite{Brug1}.\footnote{In the course
    of this work, simple objects in the M\"{u}ger center will be indexed by
    Greek letters to distinguish them from simple objects in $\mcC$ which will be
    indexed by lower case Latin letters.} This center constitutes a full symmetric
    ribbon subcategory of $\mcC$ which is trivial if and only if $\mcC$ is modular.
    In fact, if $\mcC$ is not modular then some column of the $S$--matrix is a
    multiple of the first \cite{Brug1}. Thus a premodular category $\mcC$ is
    symmetric if $\mcC=\mcC'$, $\mcC$ is modular if $\mcC'=\lcb\mbbI\rcb$, and
    $\mcC$ is properly premodular otherwise.

    Given these abstract constructions one might wonder if premodular categories
    exist and indeed they do; for instance, quantum groups lead not only to
    modular, but also to properly premodular categories \cite{R1}. Given
    their existence, a classification program has been taken up. In \cite{O2},
    \cite{O6}, and \cite{O4}, Ostrik has classified all fusion categories of
    ranks 2 and 3 and all premodular categories of rank 3. However, at the time
    of this writing it is not known if there are finitely many premodular
    categories up to equivalence.

    Such a problem is referred to as a \textit{rank finiteness problem}. In
    \cite{RSW} the rank finiteness problem was posed for modular categories
    while in \cite{O2} it was posed for fusion categories. Over the years
    progress has been made in various directions. For instance, direct
    classification of (pre)modular categories demonstrate the conjecture in low
    rank, while \cite{ENO1} showed rank finiteness for bounded FP-dimension and
    weakly integral categories. In a recent paper, \cite{BNRW1}, the rank
    finiteness problem was solved for modular categories. The proof for modular
    categories demonstrated connections between number theory and
    modular categories and heavily relied on the Frobenius-Schur indicators via
    the Cauchy Theorem for Modular Categories. In this paper we will extend the
    rank 4 premodular classification which depends strongly on Frobenius-Schur
    indicators. This suggests that they are fundamental to the theory of
    premodular categories.

\section{Frobenius-Schur indicators}
  \label{Section: Frobenius-Schur Indicators}
  As alluded to in the literature e.g. \cite{DGNO1}, the study of fusion
  categories is the correct generalization of the study of the representation
  theory of finite groups. Each finite group, $G$, gives rise to a fusion
  category whose objects are the representations of $G$ and whose morphisms are
  intertwiners \cite{DGNO1}. With this connection, it is natural to ask if the
  techniques used in the study of finite group representations can be
  generalized to arbitrary fusion categories and often they can. For instance,
  the class equation was generalized in \cite{ENO1}, a rigorous study of
  Frobenius-Schur indicators was undertaken in \cite{NS2}\cite{NS1}, and the
  Cauchy Theorem was fully extended to modular categories in \cite{BNRW1}.

  In the classical theory of the representations of finite groups one can form
  the $n^{\text{th}}$-Frobenius-Schur indicator from the characters for any
  $n\in\mbbN$. The $0^{\text{th}}$ Frobenius-Schur indicator gives the dimension
  of the representation, the $1^{\text{st}}$ indicator detects if the
  representation is the trivial representation. The $2^{\text{nd}}$ indicator of
  an irreducible representation is $1$, $0$, or $-1$ depending on if the
  representation is real, complex, or quaternionic. Frobenius-Schur indicators
  have also been developed for and applied to semisimple Hopf algebras
  \cite{LM,KSZ1}.

  The $2^{\text{nd}}$ Frobenius-Schur indicator in the context of fusion
  categories was first computed by physicists studying rational conformal field
  theories \cite{Ban1}. The study of Frobenius-Schur indicators was furthered
  by Siu-Hung Ng and Peter Schauenburg who applied the graphical calculus and
  categorical considerations to derive graphical expressions for the
  $n^{\text{th}}$ Frobenius-Schur indicators of pivotal, spherical, and modular
  categories. In the modular case, they recovered Bantay's result and found
  similar formulas for computing the $n^{\text{th}}$ indicator of a modular
  category in terms of the modular datum. If the modularity assumption is
  dropped it is not known how to compute the $n^{\text{th}}$ indicator strictly
  in terms of the premodular datum; that is without recourse to the graphical
  calculus. In this section, we will determine the following formula for the
  $2^{\text{nd}}$ Frobenius-Schur indicator of a premodular category: 
  \begin{align*}
    \n_{2}\(X_{a}\)&=\frac{1}{D^{2}}\sum_{b,c}N_{bc}^{a}d_{b}d_{c}\(\frac{\th_{b}}{\th_{c}}\)^{2}-\th_{a}\sum_{\g\in\mcC'\setminus\mbbI}d_{\g}\Tr\(R_{\g}^{aa}\).
  \end{align*}

  If the modularity condition is enforced, one sees that $\mcC'=\lcb\mbbI\rcb$
  and so the above formula recovers Bantay's result.

  Examination of Ng and Schauenburg's proof presented in \cite{NS2} reveals that
  modularity is only used indirectly when invoking \cite[Corollary 3.1.11]{BKi}.
  This corollary can be modified to give a starting place for computing the
  $2^{\text{nd}}$ indicator in the premodular setting.
  \begin{prop}
    \label{DoubleLoopRemoval}
    If $\mcC$ is premodular and $X_{a}$ is self-dual then
    \begin{equation*}
      \frac{d_{a}}{D^{2}}\vcenter{\hbox{
      \begin{tikzpicture}[xscale=0.23,yscale=0.23]
        \path[draw=black] ({2*cos(130)},{sin(130)}) arc (130:410:2 and 1);
        \path[draw=black, middlearrow={stealth}] ({2*cos(110)},{sin(110)-1.5}) --({2*cos(110)},5);
        \path[draw=black] ({2*cos(110)},{sin(110)-2.5}) --({2*cos(110)},-6);
        \path[draw=black, middlearrow={stealth}] ({2*cos(430)},{sin(430)-1.5}) --({2*cos(430)},5);
        \path[draw=black] ({2*cos(430)},{sin(430)-2.5}) --({2*cos(430)},-6);
        \node at ({2*cos(110)-0.5},{4.5}) {\scalebox{0.75}{$a$}};
        \node at ({2*cos(430)+0.5},{4.5}) {\scalebox{0.75}{$a$}};
      \end{tikzpicture}
      }}
      =
      \vcenter{\hbox{
      \begin{tikzpicture}[xscale=0.23,yscale=0.23]
        \path[draw=black] (-1,2) arc (180:360:1 and 1);
        \path[draw=black, middlearrow={stealth}] (-1,2) -- (-1,5);
        \path[draw=black, middlearrow={stealth}] (1,2) -- (1,5);
        \path[draw=black] (1,-3) arc (0:180:1 and 1);
        \path[draw=black, middlearrow={stealth}] (-1,-6) -- (-1,-3);
        \path[draw=black, middlearrow={stealth}] (1,-6) -- (1,-3);
        \node at (-1.5,4.5) {\scalebox{0.75}{$a$}};
        \node at (1.5,4.5) {\scalebox{0.75}{$a$}};
        \node at (-1.5,-5) {\scalebox{0.75}{$a$}};
        \node at (1.5,-5) {\scalebox{0.75}{$a$}};
      \end{tikzpicture}
      }}
      +\sum_{\g\in\mcC'\setminus\mbbI,i}\sqrt{d_{\g}}
      \vcenter{\hbox{
      \begin{tikzpicture}[xscale=0.23,yscale=0.23]
        \path[draw=black, middlearrow={stealth}] (0,-3)--(0,2);
        \path[draw=black] (-1,3) arc (180:360:1 and 1);
        \path[draw=black, middlearrow={stealth}] (-1,3) -- (-1,5);
        \path[draw=black, middlearrow={stealth}] (1,3) -- (1,5);
        \path[draw=black] (1,-4) arc (0:180:1 and 1);
        \path[draw=black, middlearrow={stealth}] (-1,-6) -- (-1,-4);
        \path[draw=black, middlearrow={stealth}] (1,-6) -- (1,-4);
        \node at (0.5,0) {\scalebox{0.75}{$\g$}};
        \node at (-1.5,4.5) {\scalebox{0.75}{$a$}};
        \node at (1.5,4.5) {\scalebox{0.75}{$a$}};
        \node at (-1.5,-5) {\scalebox{0.75}{$a$}};
        \node at (1.5,-5) {\scalebox{0.75}{$a$}};
        \node[draw=black, fill=white, shape=circle, inner sep=0.5pt, minimum width=10pt] at (0,2) {\scalebox{0.75}{$i$}};
        \node[draw=black, fill=white, shape=circle, inner sep=0.5pt, minimum width=10pt] at (0,-3) {\scalebox{0.75}{$j$}};
      \end{tikzpicture}
      }}
    \end{equation*}
  \end{prop}
  \begin{proof}
    Applying equation \eqref{SplitThenFuse} and \cite[Lemma 3.1.4]{BKi} we have
    \begin{align*}
      \vcenter{\hbox{
      \begin{tikzpicture}[xscale=0.23,yscale=0.23]
        \path[draw=black] ({2*cos(130)},{sin(130)}) arc (130:410:2 and 1);
        \path[draw=black, middlearrow={stealth}] ({2*cos(110)},{sin(110)-1.5}) --({2*cos(110)},5);
        \path[draw=black] ({2*cos(110)},{sin(110)-2.5}) --({2*cos(110)},-6);
        \path[draw=black, middlearrow={stealth}] ({2*cos(430)},{sin(430)-1.5}) --({2*cos(430)},5);
        \path[draw=black] ({2*cos(430)},{sin(430)-2.5}) --({2*cos(430)},-6);
        \node at ({2*cos(110)-0.5},{4.5}) {\scalebox{0.75}{$a$}};
        \node at ({2*cos(430)+0.5},{4.5}) {\scalebox{0.75}{$a$}};
      \end{tikzpicture}
      }}
      &=\sum_{b,c,i}\frac{d_{b}d_{c}}{\th\(a,a,c\)}
      \vcenter{\hbox{
      \begin{tikzpicture}[xscale=0.23,yscale=0.23]
        \path[draw=black, middlearrow={stealth reversed}] ({2*cos(110)},{sin(110)}) arc (110:430:2 and 1);
        \path[draw=black, middlearrow={stealth}] (0,-0.7)--(0,2);
        \path[draw=black] (0,-1.3)--(0,-3);
        \path[draw=black] (-1,3) arc (180:360:1 and 1);
        \path[draw=black, middlearrow={stealth}] (-1,3) -- (-1,5);
        \path[draw=black, middlearrow={stealth}] (1,3) -- (1,5);
        \path[draw=black] (1,-4) arc (0:180:1 and 1);
        \path[draw=black, middlearrow={stealth}] (-1,-6) -- (-1,-4);
        \path[draw=black, middlearrow={stealth}] (1,-6) -- (1,-4);
        \node at (0.5,0) {\scalebox{0.75}{$c$}};
        \node at (-2.5,0) {\scalebox{0.75}{$b$}};
        \node at (-1.5,4.5) {\scalebox{0.75}{$a$}};
        \node at (1.5,4.5) {\scalebox{0.75}{$a$}};
        \node at (-1.5,-5) {\scalebox{0.75}{$a$}};
        \node at (1.5,-5) {\scalebox{0.75}{$a$}};
        \node[draw=black, fill=white, shape=circle, inner sep=0.5pt, minimum width=10pt] at (0,2) {\scalebox{0.75}{$i$}};
        \node[draw=black, fill=white, shape=circle, inner sep=0.5pt, minimum width=10pt] at (0,-3) {\scalebox{0.75}{$j$}};
      \end{tikzpicture}
      }}
      =\sum_{b,c,i}\frac{d_{b}d_{c}}{\th\(a,a,c\)}\frac{\tilde{s}_{bc}}{d_{c}}
      \vcenter{\hbox{
      \begin{tikzpicture}[xscale=0.23,yscale=0.23]
        \path[draw=black, middlearrow={stealth}] (0,-3)--(0,2);
        \path[draw=black] (-1,3) arc (180:360:1 and 1);
        \path[draw=black, middlearrow={stealth}] (-1,3) -- (-1,5);
        \path[draw=black, middlearrow={stealth}] (1,3) -- (1,5);
        \path[draw=black] (1,-4) arc (0:180:1 and 1);
        \path[draw=black, middlearrow={stealth}] (-1,-6) -- (-1,-4);
        \path[draw=black, middlearrow={stealth}] (1,-6) -- (1,-4);
        \node at (0.5,0) {\scalebox{0.75}{$c$}};
        \node at (-1.5,4.5) {\scalebox{0.75}{$a$}};
        \node at (1.5,4.5) {\scalebox{0.75}{$a$}};
        \node at (-1.5,-5) {\scalebox{0.75}{$a$}};
        \node at (1.5,-5) {\scalebox{0.75}{$a$}};
        \node[draw=black, fill=white, shape=circle, inner sep=0.5pt, minimum width=10pt] at (0,2) {\scalebox{0.75}{$i$}};
        \node[draw=black, fill=white, shape=circle, inner sep=0.5pt, minimum width=10pt] at (0,-3) {\scalebox{0.75}{$j$}};
      \end{tikzpicture}
      }}\\
      &=\sum_{c,i}\frac{\(\tilde{s}^{2}\)_{0c}}{\th\(a,a,c\)}
      \vcenter{\hbox{
      \begin{tikzpicture}[xscale=0.23,yscale=0.23]
        \path[draw=black, middlearrow={stealth}] (0,-3)--(0,2);
        \path[draw=black] (-1,3) arc (180:360:1 and 1);
        \path[draw=black, middlearrow={stealth}] (-1,3) -- (-1,5);
        \path[draw=black, middlearrow={stealth}] (1,3) -- (1,5);
        \path[draw=black] (1,-4) arc (0:180:1 and 1);
        \path[draw=black, middlearrow={stealth}] (-1,-6) -- (-1,-4);
        \path[draw=black, middlearrow={stealth}] (1,-6) -- (1,-4);
        \node at (0.5,0) {\scalebox{0.75}{$c$}};
        \node at (-1.5,4.5) {\scalebox{0.75}{$a$}};
        \node at (1.5,4.5) {\scalebox{0.75}{$a$}};
        \node at (-1.5,-5) {\scalebox{0.75}{$a$}};
        \node at (1.5,-5) {\scalebox{0.75}{$a$}};
        \node[draw=black, fill=white, shape=circle, inner sep=0.5pt, minimum width=10pt] at (0,2) {\scalebox{0.75}{$i$}};
        \node[draw=black, fill=white, shape=circle, inner sep=0.5pt, minimum width=10pt] at (0,-3) {\scalebox{0.75}{$j$}};
      \end{tikzpicture}
      }}
      =\vcenter{\hbox{
      \begin{tikzpicture}[xscale=0.23,yscale=0.23]
        \path[draw=black] (-1,2) arc (180:360:1 and 1);
        \path[draw=black, middlearrow={stealth}] (-1,2) -- (-1,5);
        \path[draw=black, middlearrow={stealth}] (1,2) -- (1,5);
        \path[draw=black] (1,-3) arc (0:180:1 and 1);
        \path[draw=black, middlearrow={stealth}] (-1,-6) -- (-1,-3);
        \path[draw=black, middlearrow={stealth}] (1,-6) -- (1,-3);
        \node at (-1.5,4.5) {\scalebox{0.75}{$a$}};
        \node at (1.5,4.5) {\scalebox{0.75}{$a$}};
        \node at (-1.5,-5) {\scalebox{0.75}{$a$}};
        \node at (1.5,-5) {\scalebox{0.75}{$a$}};
      \end{tikzpicture}
      }}
      +\sum_{c\neq0,i}\frac{\(\tilde{s}^{2}\)_{0c}}{\th\(a,a,c\)}
      \vcenter{\hbox{
      \begin{tikzpicture}[xscale=0.23,yscale=0.23]
        \path[draw=black, middlearrow={stealth}] (0,-3)--(0,2);
        \path[draw=black] (-1,3) arc (180:360:1 and 1);
        \path[draw=black, middlearrow={stealth}] (-1,3) -- (-1,5);
        \path[draw=black, middlearrow={stealth}] (1,3) -- (1,5);
        \path[draw=black] (1,-4) arc (0:180:1 and 1);
        \path[draw=black, middlearrow={stealth}] (-1,-6) -- (-1,-4);
        \path[draw=black, middlearrow={stealth}] (1,-6) -- (1,-4);
        \node at (0.5,0) {\scalebox{0.75}{$c$}};
        \node at (-1.5,4.5) {\scalebox{0.75}{$a$}};
        \node at (1.5,4.5) {\scalebox{0.75}{$a$}};
        \node at (-1.5,-5) {\scalebox{0.75}{$a$}};
        \node at (1.5,-5) {\scalebox{0.75}{$a$}};
        \node[draw=black, fill=white, shape=circle, inner sep=0.5pt, minimum width=10pt] at (0,2) {\scalebox{0.75}{$i$}};
        \node[draw=black, fill=white, shape=circle, inner sep=0.5pt, minimum width=10pt] at (0,-3) {\scalebox{0.75}{$j$}};
      \end{tikzpicture}
      }}
    \end{align*}

    Since the columns of the columns of the $S$-matrix are eigenvectors of the
    fusion matrices we know that $\(\tilde{s}^{2}\)_{\g0}=d_{\g}D^{2}$ if
    $X_{\g}\in\mcC'$ and $0$ otherwise; this observation gives the desired result.
  \end{proof}

  Recall from \cite{NS2} that the $n^{\text{th}}$ Frobenius-Schur indicator is
  defined by $\n_{n}\(X\)=Tr\(E_{X}^{\(n\)}\)$, where $E_{X}^{\(n\)}$ is given
  by
  \begin{align*}
    E_{X}^{n}:
    \vcenter{\hbox{
    \begin{tikzpicture}[xscale=0.23,yscale=0.23]
      \path[draw=black] (-0.5,0) -- (-0.5,-5);
      \path[draw=black] (0.5,0) -- (0.5,-5);
      \path[draw=white] (3.5,2) arc (0:180:3 and 3);
      \node[draw=black, fill=white,shape=rectangle, minimum width=1cm, minimum height=0.5cm] at (0,0) {\scalebox{0.75}{$f$}};
      \node at (-1,-5.5) {\scalebox{0.75}{$V$}};
      \node at (2,-5.5) {\scalebox{0.75}{$V^{\otimes\(n-1\)}$}};
    \end{tikzpicture}
    }}
    \mapsto
    \vcenter{\hbox{
    \begin{tikzpicture}[xscale=0.23,yscale=0.23]
      \path[draw=black] (-0.5,0) -- (-0.5,-4);
      \path[draw=black] (0.5,0) -- (0.5,-5);
      \path[draw=black] (-2.5,-4) arc (180:360:1 and 1);
      \path[draw=black] (-2.5,-4) -- (-2.5, 2);
      \path[draw=black] (3.5,2) arc (0:180:3 and 3);
      \path[draw=black] (3.5,2) -- (3.5,-5);
      \node[draw=black, fill=white,shape=rectangle, minimum width=1cm, minimum height=0.5cm] at (0,0) {\scalebox{0.75}{$f$}};
      \node at (4,-5.5) {\scalebox{0.75}{$V$}};
      \node at (1.5,-5.5) {\scalebox{0.75}{$V^{\otimes\(n-1\)}$}};
    \end{tikzpicture}
    }}
  \end{align*}

  Applying techniques from \cite{NS2} and our bases for the splitting and fusion
  spaces, to this definition, we find that if $X_{a}$ is self-dual, then the
  $2^{\text{nd}}$ Frobenius-Schur indicator is given by
  \begin{align}
    \label{2nd FS Indicator Definition}
    \nu_{2}\(X_{a}\)=\frac{\th_{a}}{d_{a}}
    \vcenter{\hbox{
    \begin{tikzpicture}[xscale=0.23,yscale=0.23]
      \path[draw=black] (-2,0) arc (180:360:1 and 1);
      \path[draw=black] (-2,0)--(-2,0.5);
      \path[draw=black] (0,0)--(0,0.5);
      \path[draw=black] (-2,0.5) arc (-180:-240:2 and 1.25);
      \path[draw=black] (-1,{2.08253-0.5}) arc (300:360:2 and 1.25);
      \path[draw=black] (0,0.5) arc (0:50:2 and 1.25);
      \path[draw=black] (-2,{3.16506-0.5}) arc (180:230:2 and 1.25);
      \path[draw=black, middlearrow={stealth}] (-2,{3.16506-0.5}) -- (-2,3);
      \path[draw=black, middlearrow={stealth}] (0,{3.16506-0.5}) -- (0,3);
      \node at (-2.5,3) {\scalebox{0.75}{$a$}};
      \node at (0.5,3) {\scalebox{0.75}{$a$}};
      \path[draw=black] (0,3) arc (-180:-360:1 and 1);
      \path[draw=black] (-2,3) arc (-180:-360:3 and 3);
      \path[draw=black] (0,-3) arc (0:180:1 and 1);
      \path[draw=black] (-2,-3) -- (-2,-3.5);
      \path[draw=black] (-2,-3.5) arc (180:360:3 and 3);
      \path[draw=black] (0,-3) -- (0,-3.5);
      \path[draw=black] (0,-3.5) arc (180:360:1 and 1);
      \path[draw=black] (2,-3.5) -- (2,3);
      \path[draw=black] (4,-3.5) -- (4,3);
    \end{tikzpicture}
    }}
  \end{align}

  otherwise we define it to be zero. Here the factor $\frac{1}{d_{a}}$ appears
  due to renormalization of the basis elements of
  $\Hom_{\mcC}\(X^{\otimes2},\mbbI\)$ and $\Hom_{\mcC}\(\mbbI,X^{\otimes 2}\)$
  to have norm 1. With this definition and proposition in place we can prove the
  following theorem.

  \begin{theorem}
    \label{Second FS By Transparent}
    If $\mcC$ is a premodular category and $X_{a}$ is a simple self-dual object then 
    \begin{align*}
      \n_{2}\(X_{a}\)&=\frac{1}{D^{2}}\sum_{b,c}N_{bc}^{a}d_{b}d_{c}\(\frac{\th_{b}}{\th_{c}}\)^{2}-\th_{a}\sum_{\g\in\mcC'\setminus\mbbI}d_{\g}\Tr\(R_{\g}^{aa}\).
    \end{align*}
  \end{theorem}
  \begin{proof}
    The proof proceeds by applying \propref{DoubleLoopRemoval} to equation
    \eqref{2nd FS Indicator Definition} and then making use of the graphical
    calculus. To simplify notation we observe that since $X_{a}$ is self-dual the
    arrow on the ribbon corresponding to this object can be safely removed.
    \begin{align*}
      \nu_{2}\(X_{a}\)&=\frac{\th_{a}}{d_{a}}
      \vcenter{\hbox{
      \begin{tikzpicture}[xscale=0.23,yscale=0.23]
        \path[draw=black] (-2,0) arc (180:360:1 and 1);
        \path[draw=black] (-2,0)--(-2,0.5);
        \path[draw=black] (0,0)--(0,0.5);
        \path[draw=black] (-2,0.5) arc (-180:-240:2 and 1.25);
        \path[draw=black] (-1,{2.08253-0.5}) arc (300:360:2 and 1.25);
        \path[draw=black] (0,0.5) arc (0:50:2 and 1.25);
        \path[draw=black] (-2,{3.16506-0.5}) arc (180:230:2 and 1.25);
        \path[draw=black] (-2,{3.16506-0.5}) -- (-2,3);
        \path[draw=black] (0,{3.16506-0.5}) -- (0,3);
        \path[draw=black] (0,3) arc (-180:-360:1 and 1);
        \path[draw=black] (-2,3) arc (-180:-360:3 and 3);
        \path[draw=black] (0,-3) arc (0:180:1 and 1);
        \path[draw=black] (-2,-3) -- (-2,-3.5);
        \path[draw=black] (-2,-3.5) arc (180:360:3 and 3);
        \path[draw=black] (0,-3) -- (0,-3.5);
        \path[draw=black] (0,-3.5) arc (180:360:1 and 1);
        \path[draw=black] (2,-3.5) -- (2,3);
        \path[draw=black] (4,-3.5) -- (4,3);
        \node at ({2*cos(110)-2},{4}) {\scalebox{0.75}{$a$}};
      \end{tikzpicture}
      }}
      =\frac{\th_{a}}{d_{a}}\frac{d_{a}}{D^{2}}\vcenter{\hbox{\begin{tikzpicture}[xscale=0.23,yscale=0.23]
        \path[draw=black] (-2,1) arc (-180:-240:2 and 1.25);
        \path[draw=black] (-1,2.08253) arc (300:360:2 and 1.25);
        \path[draw=black] (0,1) arc (0:50:2 and 1.25);
        \path[draw=black] (-2,3.16506) arc (180:230:2 and 1.25);
        \path[draw=black] (-2,1) -- (-2,-0.5);
        \path[draw=black] (0,1) -- (0,-0.5);
        \path[draw=black] (-2,3.16506) -- (-2,3.5);
        \path[draw=black] (0,3.16506) -- (0,3.5);
        \path[draw=black] (-2,-1.25) -- (-2,-3);
        \path[draw=black] (0,-1.25) -- (0,-3);
        \path[draw=black] ({2*cos(130)-1},{sin(130)}) arc (130:410:2 and 1);
        \path[draw=black] (0,-3) arc (180:360:1 and 1);
        \path[draw=black] (-2,-3) arc (180:360:3 and 3);
        \path[draw=black] (0,3.5) arc (-180:-360:1 and 1);
        \path[draw=black] (-2,3.5) arc (-180:-360:3 and 3);
        \path[draw=black] (2,3.5) -- (2,-3);
        \path[draw=black] (4,3.5) -- (4,-3);
        \node at ({2*cos(110)-2},{4}) {\scalebox{0.75}{$a$}};
      \end{tikzpicture}
      }}
      -
      \frac{\th_{a}}{d_{a}}\sum_{\g\in\mcC'\setminus\mbbI,i,j}\sqrt{d_{\g}}
      \vcenter{\hbox{
      \begin{tikzpicture}[xscale=0.23,yscale=0.23]
        \path[draw=black] (-2,0.8) arc (180:360:1 and 1);
        \path[draw=black] (-2,0.8) arc (-180:-240:2 and 1.25);
        \path[draw=black] (-1,{2.08253-0.2}) arc (300:360:2 and 1.25);
        \path[draw=black] (0,0.8) arc (0:50:2 and 1.25);
        \path[draw=black] (-2,{3.16506-0.2}) arc (180:230:2 and 1.25);
        \path[draw=black] (-2,{3.16506-0.2}) -- (-2,3);
        \path[draw=black] (0,{3.16506-0.2}) -- (0,3);
        \path[draw=black] (0,3) arc (-180:-360:1 and 1);
        \path[draw=black] (-2,3) arc (-180:-360:3 and 3);
        \path[draw=black] (0,-3.5) arc (0:180:1 and 1);
        \path[draw=black] (-2,-3.5) arc (180:360:3 and 3);
        \path[draw=black] (0,-3.5) arc (180:360:1 and 1);
        \path[draw=black] (2,-3.5) -- (2,3);
        \path[draw=black] (4,-3.5) -- (4,3);
        \path[draw=black, middlearrow={stealth}] (-1,-2.5) -- (-1,-0.25);
        \node[draw=black, fill=white, shape=circle, inner sep=0.5pt, minimum width=10pt] at (-1,-0.1) {\scalebox{0.75}{$i$}};
        \node[draw=black, fill=white, shape=circle, inner sep=0.5pt, minimum width=10pt] at (-1,-2.7) {\scalebox{0.75}{$j$}};
        \node at ({2*cos(110)-2},{4}) {\scalebox{0.75}{$a$}};
        \node at (-0.25,-1.5) {\scalebox{0.75}{$\g$}};
      \end{tikzpicture}
      }}\\
      &=\frac{\th_{a}}{D^{2}}\sum_{b}d_{b}\vcenter{\hbox{
      \begin{tikzpicture}[xscale=0.23,yscale=0.23]
        \node at ({2*cos(150)-0.5},{sin(150)}) {\scalebox{0.75}{$b$}};
        \path[draw=black, middlearrow={stealth reversed}] ({2*cos(150)},{sin(150)}) arc (160:270:2 and 1);
        \path[draw=black] ({2*cos(130)},{sin(190)}) -- ({2*cos(130)},3.5);
        \path[draw=black] ({2*cos(130)},3.5) arc(-180:-360:1 and 1);
        \path[draw=black] ({2*cos(130)+2},3.5)--({2*cos(130)+2},-4);
        \path[draw=black] ({2*cos(130)+2},-4) arc (180:360:1 and 1);
        \path[draw=black] ({2*cos(130)+2.5},-1)--({2*cos(130)+4.5},-1);
        \path[draw=black] ({2*cos(130)+4.5},-1) arc (270:450:2 and 1);
        \path[draw=black] ({2*cos(130)+4},-0.5)--({2*cos(130)+4},2.5);
        \path[draw=black] ({2*cos(130)+4},2.5) arc (-180:-360:2 and 2);
        \path[draw=black] ({2*cos(130)+8},2.5)--({2*cos(130)+8},-4);
        \path[draw=black] ({2*cos(130)},-4) arc (180:360:4 and 4);
        \path[draw=black] ({2*cos(130)},-4) -- ({2*cos(130)},{sin(190)-0.5});
        \path[draw=black] ({2*cos(130)+3.5},1) -- ({2*cos(130)+2.75},1);
        \path[draw=black] (0,1) arc (90:120:2 and 1);
        \path[draw=black] ({2*cos(130)+4},-3) -- ({2*cos(130)+4},-1.5);
        \path[draw=black] ({2*cos(130)+5},-3) arc (90:180:1 and 1);
        \path[draw=black] ({2*cos(130)+5},-4) arc (-90:90:0.5 and 0.5);
        \path[draw=black] ({2*cos(130)+5},-4) arc (-90:-130:1 and 1);
        \path[draw=black] ({2*cos(130)+4},-3) arc (180:200:1 and 1);
        \node at ({2*cos(110)-1},{3}) {\scalebox{0.75}{$a$}};
      \end{tikzpicture}
      }}
      -
      \frac{\th_{a}}{d_{a}}\sum_{\g\in\mcC'\setminus\mbbI,i,j}\sqrt{d_{\g}}R_{\g,i}^{aa}
      \vcenter{\hbox{
      \begin{tikzpicture}[xscale=0.23,yscale=0.23]
        \path[draw=black] (-2,2.75) arc (180:360:1 and 1);
        \path[draw=black] (-2,2.75) -- (-2,3.25);
        \path[draw=black] (0,2.75) -- (0,3.25);
        \path[draw=black] (0,3.25) arc (-180:-360:1 and 1);
        \path[draw=black] (-2,3.25) arc (-180:-360:3 and 3);
        \path[draw=black] (0,-2.75) arc (0:180:1 and 1);
        \path[draw=black] (-2,-3.25) arc (180:360:3 and 3);
        \path[draw=black] (0,-3.25) arc (180:360:1 and 1);
        \path[draw=black] (2,-3.25) -- (2,3.25);
        \path[draw=black] (4,-3.25) -- (4,3.25);
        \path[draw=black] (-2,-2.75) -- (-2,-3.25);
        \path[draw=black] (0,-2.75) -- (0,-3.25);
        \path[draw=black, middlearrow={stealth}] (-1,-1.75) -- (-1,1.75);
        \node[draw=black, fill=white, shape=circle, inner sep=0.5pt, minimum width=10pt] at (-1,1.75) {\scalebox{0.75}{$i$}};
        \node[draw=black, fill=white, shape=circle, inner sep=0.5pt, minimum width=10pt] at (-1,-1.75) {\scalebox{0.75}{$j$}};
        \node at (-0.25,-0.5) {\scalebox{0.75}{$\g$}};
        \node at ({2*cos(110)-2},{3}) {\scalebox{0.75}{$a$}};
        \node at ({2*cos(110)+1.25},{3}) {\scalebox{0.75}{$a$}};
      \end{tikzpicture}
      }}
    \end{align*}
    \begin{align*}
      \phantom{\n_{2}\(X_{a}\)}&=\frac{\th_{a}^{2}}{D^{2}}\sum_{b}d_{b}\vcenter{\hbox{
      \begin{tikzpicture}[xscale=0.23,yscale=0.23]
        \node at ({2*cos(150)-0.5},{sin(150)}) {\scalebox{0.75}{$b$}};
        \path[draw=black, middlearrow={stealth reversed}] ({2*cos(150)},{sin(150)}) arc (160:270:2 and 1);
        \path[draw=black] ({2*cos(130)},{sin(190)}) -- ({2*cos(130)},2);
        \path[draw=black] ({2*cos(130)},2) arc(-180:-360:1 and 1);
        \path[draw=black] ({2*cos(130)+2},2)--({2*cos(130)+2},-1.5);
        \path[draw=black] ({2*cos(130)+2},-1.5) arc (180:360:1 and 1);
        \path[draw=black] ({2*cos(130)+2.5},-1)--({2*cos(130)+4.5},-1);
        \path[draw=black] ({2*cos(130)+4.5},-1) arc (270:450:2 and 1);
        \path[draw=black] ({2*cos(130)+4},-0.5)--({2*cos(130)+4},2);
        \path[draw=black] ({2*cos(130)+4},2) arc (-180:-360:2 and 2);
        \path[draw=black] ({2*cos(130)+8},2)--({2*cos(130)+8},-0.5);
        \path[draw=black] ({2*cos(130)},-0.5) arc (180:360:4 and 4);
        \path[draw=black] ({2*cos(130)+3.5},1) -- ({2*cos(130)+2.75},1);
        \path[draw=black] (0,1) arc (90:120:2 and 1);
        \node at ({2*cos(150)-0.25},{2}) {\scalebox{0.75}{$a$}};
       \end{tikzpicture}
       }}
      -
      \frac{\th_{a}}{d_{a}}\sum_{\g\in\mcC'\setminus\mbbI,i,j}\sqrt{d_{\g}}R_{\g,i}^{aa}\th\(a,a,\g\)\delta_{ij}\\
      &=\frac{\th_{a}^{2}}{D^{2}}\sum_{b,c,i,j}\frac{d_{b}d_{c}}{\th\(a,b,c\)}\vcenter{\hbox{\begin{tikzpicture}[xscale=0.23,yscale=0.23]
        \path[draw=black] (-2,1) arc (-180:-240:2 and 1.25);
        \path[draw=black] (-1,2.08253) arc (300:360:2 and 1.25);
        \path[draw=black] (0,1) arc (0:50:2 and 1.25);
        \path[draw=black] (-2,3.16506) arc (180:230:2 and 1.25);
        \path[draw=black] (-2,{3.16506*1}) arc (-180:-240:2 and 1.25);
        \path[draw=black] (-1,{2.08253-1+3.16506*1}) arc (300:360:2 and 1.25);
        \path[draw=black] (0,{3.16506*1}) arc (0:50:2 and 1.25);
        \path[draw=black] (-2,{3.16506-1+3.16506*1}) arc (180:230:2 and 1.25);
        \path[draw=black] (-2,{3.16506*2-1}) arc (-180:-240:2 and 1.25);
        \path[draw=black] (-1,{2.08253+2.16506*2}) arc (300:360:2 and 1.25);
        \path[draw=black] (0,{3.16506*2-1}) arc (0:50:2 and 1.25);
        \path[draw=black] (-2,{3.16506+2.16506*2}) arc (180:230:2 and 1.25);
        \path[draw=black] (-2,{3.16506*3-2}) arc (-180:-240:2 and 1.25);
        \path[draw=black] (-1,{2.08253+2.16506*3}) arc (300:360:2 and 1.25);
        \path[draw=black] (0,{3.16506*3-2}) arc (0:50:2 and 1.25);
        \path[draw=black] (-2,{3.16506+2.16506*3}) arc (180:230:2 and 1.25);
        \path[draw=black] (-2,0) -- (-2,1);
        \path[draw=black] (0,0) -- (0,1);
        \path[draw=black] (-2,0) arc (180:360:3 and 3);
        \path[draw=black] (0,0) arc (180:360:1 and 1);
        \path[draw=black] (-2,{3.16506+2.16506*3}) -- (-2,10);
        \path[draw=black] (0,{3.16506+2.16506*3}) -- (0,10);
        \path[draw=black] (-2,10) arc (-180:-360:3 and 3);
        \path[draw=black] (0,10) arc (-180:-360:1 and 1);
        \path[draw=black, middlearrow={stealth reversed}] (2,10) -- (2,8);
        \path[draw=black] (4,10) -- (4,8);
        \path[draw=black, middlearrow={stealth}] (2,0) -- (2,2);
        \path[draw=black] (4,0) -- (4,2);
        \path[draw=black] (4,2) arc (0:180:1 and 1);
        \path[draw=black] (2,8) arc (180:360:1 and 1);
        \path[draw=black, middlearrow={stealth}] (3,3) -- (3,7);
        \node[draw=black, fill=white, shape=circle, inner sep=0.5pt, minimum width=10pt] at (3,7) {\scalebox{0.75}{$i$}};
        \node[draw=black, fill=white, shape=circle, inner sep=0.5pt, minimum width=10pt] at (3,3) {\scalebox{0.75}{$j$}};
        \node at (4.5,2) {\scalebox{0.75}{$a$}};
        \node at (1.5,2) {\scalebox{0.75}{$b$}};
        \node at (4.5,8) {\scalebox{0.75}{$a$}};
        \node at (1.5,8) {\scalebox{0.75}{$b$}};
        \node at (3.5,4.5) {\scalebox{0.75}{$c$}};
      \end{tikzpicture} 
      }}
      -\th_{a}\sum_{\g\in\mcC'\setminus\mbbI}d_{\g}\Tr\(R_{\g}^{aa}\)\\
      &=\frac{\th_{a}^{2}}{D^{2}}\sum_{b,c,i,j}\frac{d_{b}d_{c}\(R_{c,i}^{ab}R_{c,i}^{ba}\)^{2}}{\th\(a,b,c\)}\th\(a,b,c\)\d_{ij}
      -\th_{a}\sum_{\g\in\mcC'\setminus\mbbI}d_{\g}\Tr\(R_{\g}^{aa}\)\\
      &=\frac{\th_{a}^{2}}{D^{2}}\sum_{b,c,i}d_{b}d_{c}\(R_{c,i}^{ab}R_{c,i}^{ba}\)^{2}
      -\th_{a}\sum_{\g\in\mcC'\setminus\mbbI}d_{\g}\Tr\(R_{\g}^{aa}\)
    \end{align*}

    Applying equation (216) of Appendix E in \cite{K1} and noting that
    $\(\tilde{s}^{2}\)_{\g0}=d_{\g}D^{2}$ for $X_{\g}\in\mcC'$ gives
    \begin{align*}
      \n_{2}\(X_{a}\)&=\frac{\th_{a}^{2}}{D^{2}}\sum_{b,c,i}d_{b}d_{c}\(\frac{\th_{c}}{\th_{a}\th_{b}}\)^{2}
      -\th_{a}\sum_{\g\in\mcC'\setminus\mbbI}d_{\g}\Tr\(R_{\g}^{aa}\).
    \end{align*}

    Making use of equation \eqref{FusionSymmetries} we have
    $N_{ab}^{c}=N_{ba}^{c}=N_{bc^{*}}^{a}=N_{c^{*}b}^{a}$. However,
    $\th_{b^{*}}=\th_{b}$ and $d_{b^{*}}=d_{b}$ so
    \begin{align*}
      \n_{2}\(X_{a}\)&=\frac{1}{D^{2}}\sum_{b,c}N_{bc^{*}}^{a}d_{b}d_{c^{*}}\(\frac{\th_{c^{*}}}{\th_{b}}\)^{2}
      -\th_{a}\sum_{\g\in\mcC'\setminus\mbbI}d_{\g}\Tr\(R_{\g}^{aa}\).
    \end{align*}

    Reindexing the first sum gives the desired result.
  \end{proof}

  Since the $R$-matrices appear in this indicator, it is of limited
  computational use. However, one can show that the two sums of \thmref{Second
  FS By Transparent} are both rational integers. To do this, we first recall
  that the M\"{u}ger center of $\mcC$ is a ribbon fusion category over $\mbbC$
  with fusion rules and twists descending from $\mcC$. Moreover, $c_{W,V}\circ
  c_{V,W}=id_{V\otimes W}$ on $\mcC'$ by its definition. So applying
  \cite[Proposition 6.1]{NS2}, we can deduce that if $X_{\g}\in\mcC'$ then
  $\th_{\g}=\pm1$. However, $\th_{a}R^{aa}_{c,i}=\pm\sqrt{\th_{c}}$ and so, if
  $X_{\g}\in\mcC'$, we deduce that $\th_{a}R^{aa}_{\g,i}\in\lcb\pm1,\pm i\rcb$,
  which leads to the following corollary.
  \begin{corollary}
    \label{Modular FS is Algebraic Integer}
    If $\mcC$ is premodular and $X_{a}\in\mcC$ simple, then 
    \begin{align*}
      \frac{1}{D^{2}}\sum_{b,c}N_{bc}^{a}d_{b}d_{c}\(\frac{\th_{b}}{\th_{c}}\)^{2}
    \end{align*}
    is real and if $X_{a}$ is self-dual then it is a rational integer.
  \end{corollary}
  \begin{proof}
    Applying \cite{NS2}, we know that $\n_{2}\(X_{a}\)\in\lcb-1,0,1\rcb$.
    Coupling this observation with the aforementioned fact that
    $\th_{a}R^{aa}_{\g,i}\in\lcb \pm1,\pm i\rcb$ for $X_{\g}\in\mcC'$, we can
    conclude that
    \begin{align*}
      \frac{1}{D^{2}}\sum_{b,c}N_{bc}^{a}d_{b}d_{c}\(\frac{\th_{b}}{\th_{c}}\)^{2}\in\mbbZ\[i\].
    \end{align*}

    However, $N_{bc}^{a}=N_{cb}^{a}$, $d_{b}\in\mbbR$, and
    $\bar{\th_{b}}=\th_{b}^{-1}$ for all $a,b,c$. So for any $a$ we have that
    $\frac{1}{D^{2}}\sum_{b,c}N_{ab}^{c}d_{b}d_{c}\(\frac{\th_{c}}{\th_{b}}\)^{2}$
    is invariant under complex conjugation. Consequently
    $\frac{1}{D^{2}}\sum_{b,c}N_{bc}^{a}d_{b}d_{c}\(\frac{\th_{b}}{\th_{c}}\)^{2}\in\mbbZ\[i\]\cap\mbbR=\mbbZ$.
  \end{proof}
  \begin{remark}
    One can apply this corollary to show that the M\"{u}ger center of a
    premodular category is integral as follows. Recall from \cite[Section
    6]{NS2}, that if $\a,\b\in\mcC'$ then
    $\th_{\a\otimes\b}=\th_{\a}\otimes\th_{\b}$ so $\th_{\a\otimes
    \b}^{2}=\th_{\a}^{2}\otimes\th_{\b}^{2}$. Consequently,
    $\sum_{\b\in
    C'}N_{\a\g}^{\b}\th_{\b}^{2}d_{\b}=\th_{\a}^{2}\th_{\g}^{2}d_{\a}d_{\g}$
    which can be rearranged to give
    \begin{align*}
      \sum_{\b\in\mcC'}N_{\a\g}^{\b}d_{\b}d_{\g}\(\frac{\th_{\b}}{\th_{\g}}\)^{2}=\th_{\a}^{2}d_{\a}d_{\g}^{2}.
    \end{align*}

    Summing over $\g\in\mcC'$ and reindexing gives
    \begin{align*}
      \th_{\a}d_{\a}&=\frac{1}{D_{\mcC'}^{2}}\sum_{\b,\g\in\mcC'}N_{\b\g}^{\a}d_{\b}d_{\g}\(\frac{\th_{\b}}{\th_{\g}}\)^{2}\in\mbbZ.
    \end{align*}

    This is equivalent to saying that the M\"{u}ger center is an integral
    subcategory of $\mcC$. Since the M\"{u}ger center is a symmetric category
    and hence necessarily Grothendieck-equivalent to a representation category of a finite group, we know
    that it is integral. However, this does provide a new (to this author) route
    to this result.
  \end{remark}

  Examination of \thmref{Second FS By Transparent} reveals that $R^{aa}_{c}$
  enters into the formula for the second indicator. Since the $R$-matrices
  involve square roots of the twists, we have that $R^{ab}_{c}$ is a
  $2N^{\text{th}}$ root of unity where $N=\ord\(T\)$. Coupling this observation
  with Frobenius-Schur exponent of \cite{NS2} motivates the following
  conjecture.
  \begin{conjecture}
    \label{CyclotomicDimensionsConjecture}
    If $\mcC$ is premodular, $X_{a}$ is a simple object and $N=\ord\(T\)$, then $d_{a}\in\mbbZ\[\z_{2N}\]$.
  \end{conjecture}

  This result is reminiscent of the Ng-Schauenburg Theorem for modular
  categories, which tells us that for any simple object $X_{a}$,
  $d_{a}\in\mbbZ\[\z_{N}\]$ where $N=\ord\(T\)$ \cite{NS2}. One might wonder if
  this theorem holds in the premodular setting despite the appearance of the
  $R$-matrices. However, examination of the premodular category
  $\mcC\(sl\(2\),8\)_{ad}$ reveals that the Ng-Schauenburg Theorem fails, but
  that \conjref{CyclotomicDimensionsConjecture} holds. Preliminary results
  indicate that more complicated combinations of the $R$-matrices may appear in
  higher indicators so more work is needed before the techniques of Ng and
  Schauenburg can be applied to \conjref{CyclotomicDimensionsConjecture}.
  However, this conjecture has been verified for premodular categories of rank
  $<5$.

\section{Rank 4 premodular categories}
  To classify all rank 4 premodular categories, we would need to determine the
  premodular datum -- $\(\tilde{S},T,N_{0},\ldots, N_{n}\)$ -- in addition to the
  $R$-- and $F$--matrices. However, Ocneanu Rigidity tells us that there are only
  finitely many braided fusion categories realizing a given fusion ring and so it
  suffices to understand only the premodular datum. When classifying modular
  categories, one has a full range of Galois techniques available in addition to
  the divisibility of dimensions and the universal grading group. However, in the
  premodular setting, all of these techniques fail. Indeed, examination of
  $\mcC\(sl\(2\),8\)_{ad}$ reveals that the universal grading group need not be
  isomorphic to $\mcC_{pt}$, the full subcategory generated by the invertible
  objects. This category further illustrates that the Ng-Schauenburg Theorem
  fails.\footnote{The dimensions of the simple objects need not live in the
  cyclotomic extension of $\mbbQ$ generated by the twists.} If we instead consider
  $\mcC\(sl\(2\),6\)_{ad}$, then we see that the square of the dimensions of the
  simple objects need not divide the categorical dimension. Finally, the tensor
  category $\Fib\times\Rep\(\mbbZ_{2}\)$ reveals that the Galois techniques fail
  in the premodular setting.

  Given the  failure of many of the techniques used in modular classification,
  what is left? To perform low rank premodular classification, people have, in
  the past, examined the double $\mcZ\(\mcC\)$ as a module category \cite{O2}.
  However, in the rank 4 case, this approach is infeasible due to the number of
  simple objects. To overcome these difficulties, we will make use of the
  equations governing the premodular datum as well as cyclotomic and number
  theoretic techniques; the minimal modularization developed by Brugui\`{e}res; and
  the $2^{\text{nd}}$ Frobenius-Schur indicators.

  Recalling our partition of premodular categories into symmetric, properly
  premodular, and modular, we will discuss each of these classes in turn. We begin
  with the symmetric case, which is readily dealt with using the classification
  due to \cite{D1}.
  \begin{proposition}
    If $\mcC$ is a rank 4 symmetric category, then it is Grothendieck equivalent
    to $\Rep\(G\)$ where $G$ is $\mbbZ/4\mbbZ$, $\mbbZ/2\mbbZ\times\mbbZ/2\mbbZ$, $D_{10}$,
    or $\mfA_{4}$. 
  \end{proposition}

  Continuing onto the well understood setting of modular categories. We recall
  that much of the classification has been completed in \cite{RSW}. The omissions
  will be filled in and the classification completed in the following
  result.\footnote{The author would like to thank Eric Rowell for suggesting this
  approach.}
  \begin{proposition}
    \label{Rank4Modular}
    If $\mcC$ is a rank 4 modular category then it is Galois conjugate to a
    modular category from \cite{RSW} or has $S$--matrix
    \begin{align*}
    \(\begin{smallmatrix}
    1&-1&\bar{\t}&\t\\
    -1&1&-\t&-\bar{\t}\\
    \bar{\t}&-\t&-1&-1\\
    \t&-\bar{\t}&-1&-1
    \end{smallmatrix}\),
    \end{align*}
    where $\t=\frac{1+\sqrt{5}}{2}$ is the golden mean and
    $\bar{\t}=\frac{1-\sqrt{5}}{2}$ is its Galois conjugate.
  \end{proposition}
  \begin{proof}
    Using an argument due to V. Ostrik, \cite[Appendix A]{HR1}, it suffices to
    consider Galois groups such that the column of the S-matrix corresponding to
    the FP-dimension and the 0-column reside in distinct Galois orbits and
    neither are fixed. Since the Galois group of a rank 4 modular category is an
    abelian subgroup of $\mfS_{4}$, we see that, up to relabelling, the only
    Galois group that we need to consider is $\langle\(0,1\)\(2,3\)\rangle$.
    This is precisely case 5 of \cite{RSW}. Applying the standard Galois
    techniques present in \cite{RSW} leads to\footnote{Here we index from 0
    rather than 1 as in \cite{RSW}.}
    \begin{align*}
      \tilde{S}&=\(\begin{smallmatrix}
      1&d_{1}&d_{2}&d_{3}\\
      d_{1}&\e_{0}&\e_{3}d_{3}&\e_{0}\e_{3}d_{2}\\
      d_{2}&\e_{3}d_{3}&s_{22}&s_{23}\\
      d_{3}&\e_{0}\e_{3}d_{2}&s_{23}&\e_{0}s_{22}
      \end{smallmatrix}\).
    \end{align*}

    Where $\e_{j}^{2}=1$ for all $j$, $d_{j}$ are the categorical dimensions,
    and $s_{22}$ and $s_{23}$ are unknown S-matrix entries. Since $\e_{0}=\pm1$
    we consider these two cases separately.

    \textbf{Case 1:} $\e_{0}=1$.\ \\
      Orthogonality of the first two columns of $\tilde{S}$ gives
      $d_{1}=-\e_{3}d_{2}d_{3}$. Applying our Galois element to this equation
      gives that $\e_{3}=-1$. Next, orthogonality of the last column with the
      others gives us that $\tilde{s}_{23}\tilde{s}_{22}=-d_{2}d_{3}$ and $\tilde{s}_{22}=-1$ or
      $\tilde{s}_{22}=d_{3}^{2}$. We now examine these two subcases separately.

    \textbf{Case 1.1:} $s_{22}=d_{3}^{2}$\ \\
      Applying the orthogonality of the first and the fourth columns of the
      $S$--matrix we find that $d_{3}=\pm d_{2}$, we can apply the Verlinde
      formula and this relation to compute $N_{11}^{3}=d_{3}-\frac{1}{d_{3}}$ and
      so $d_{3}=\frac{n\pm\sqrt{4+n^{2}}}{2}$ for some $n\in\mbbN$. Examining the
      remaining $N_{1j}^{k}$ we find that either $n=0$ or $d_{2}=d_{3}$. However,
      if $n=0$, we have $d_{a}=\pm1$ for all $a$. Since rank 4 pointed modular
      categories have been classified we may assume $d_{2}=d_{3}$. Under this
      assumption the $S$--matrix takes the form
      \begin{equation*}
        \tilde{S}=\(\begin{smallmatrix}
        1&d_{3}^{2}&d_{3}&d_{3}\\
        d_{3}^{2}&1&-d_{3}&-d_{3}\\
        d_{3}&-d_{3}&d_{3}^{2}&-1\\
        d_{3}&-d_{3}&-1&d_{3}^{2}
        \end{smallmatrix}\).
      \end{equation*}

      Applying the balancing relation-- equation \eqref{Balancing}, and the
      Verlinde formula, we find\ \\
      $-1=\tilde{s}_{23}=\frac{\(n\pm\sqrt{4+n^{2}}\)^{2}\th_{1}}{4\th_{2}\th_{3}}$.
      Taking the modulus of both sides and recalling that $\abs{\th_{a}}=1$ gives
      the equation $4=\(n\pm\sqrt{4+n^{2}}\)^{2}$, whose only solution over
      $\mbbN$ is $n=0$ and so we have that $\mcC$ is pointed.

    \textbf{Case 1.2:} $\tilde{s}_{22}=-1$\ \\
      In this case, we apply the Verlinde formula to compute $N_{11}^{2}$ and
      $N_{11}^{3}$ which leads to\ \\ $d_{2}=\frac{1}{2}\(n\pm\sqrt{4+n^{2}}\)$
      and $d_{3}=\frac{1}{2}\(m\pm\sqrt{4+m^{2}}\)$ for some $m,n\in\mbbN$. The
      balancing equation for $\tilde{s}_{23}$ gives that $\th_{1}=\th_{2}\th_{3}$
      which then leads to 
      \begin{align*}
        d_{2}&=\pm\sqrt{\frac{-1+\th_{2}-\th_{2}^{2}}{\th_{2}}},\quad d_{3}=\pm\sqrt{\frac{-1+\th_{3}-\th_{3}^{2}}{\th_{3}}}
      \end{align*}

      by the balancing relation for $\tilde{s}_{22}$ and $\tilde{s}_{33}$.
      However, these results imply that $\th_{2}$ and $\th_{3}$ satisfy degree 4
      integral polynomials and are roots of unity. Applying the inverse Euler
      (totient) phi function, we see that $\th_{2},\th_{3}$ are $\pm i$ or
      primitive $5^{\text{th}}$ roots of unity and so $d_{2},d_{3}\in\lcb
      \pm1,\pm\t,\pm\bar{\t}\rcb$ where $\t$ is the golden mean
      $\frac{1}{2}\(1+\sqrt{5}\)$ and $\bar{\t}$ is its Galois conjugate.
      Simple computer search leads to $48$ $\(\tilde{S},T\)$ combinations. Twelve of the $S$--matrices
      are distinct with half of them Galois conjugate to the other half. Of these
      remaining six, two can be removed by relabeling. Thus, we have the following
      four $S$--matrices and their Galois conjugates:
      \begin{align*}
        \(\begin{smallmatrix}
        1&-1&\bar{\t}&\t\\
        -1&1&-\t&-\bar{\t}\\
        \bar{\t}&-\t&-1&-1\\
        \t&-\bar{\t}&-1&-1
        \end{smallmatrix}\)
        \quad
        \(\begin{smallmatrix}
        1&-1&-\t&\t\\
        -1&-1&-\t&-\t\\
        -\t&-\t&1&1\\
        \t&-\t&1&-1
        \end{smallmatrix}\)
        \quad
        \(\begin{smallmatrix}
        1&1&\t&\t\\
        1&-1&-\t&\t\\
        \t&-\t&1&-1\\
        \t&\t&-1&-1
        \end{smallmatrix}\)
        \quad
        \(\begin{smallmatrix}
        1&\t^{2}&\t&\t\\
        \t^{2}&1&-\t&-\t\\
        \t&-\t&-1&\t^{2}\\
        \t&-\t&\t^{2}&-1
        \end{smallmatrix}\).
      \end{align*}

      The second matrix can be discarded since there is no rank 2 modular category
      with $S$--matrix
      $\(\begin{smallmatrix}1&-1\\-1&-1\end{smallmatrix}\)$.\footnote{To see this,
      note that $N_{11}^{1}=0$ by dimension count and the other fusion coefficients
      are determined by equation \eqref{FusionSymmetries}. However, these fusion
      coefficients violate the Verlinde formula.} The last two matrices are
      pseudo-unitary and hence appear in \cite{RSW} which leaves only the first
      $S$--matrix which corresponds to $\Fib\boxtimes\bar{\Fib}$.

    \textbf{Case 2:} $\e_{0}=-1$\ \\
      By resolving the labeling ambiguity present between the 2 and 3 labels we
      can take $\e_{3}=1$. There are now two subcases

      \textbf{Case 2.1:} $\abs{d_{1}}\geq 1$\ \\
        Following the procedure of \cite{RSW}, we find that
        $d_{1}=\frac{1}{2}\(n\pm\sqrt{n^{2}+4}\)$ and $\exists a,b\in\mbbQ$ and
        $r,s\in\mbbZ$ such that 
        \begin{align*}
          r&=2b+an,\quad s=bn-2a,\quad\ \\
          d_{2}&=ad_{1}+b,\quad d_{3}=bd_{1}-a,\ \\
          D^{2}&=\(1+d_{1}^{2}\)\(1+a^{2}+b^{2}\).
        \end{align*} 

        Additionally, their techniques lead to $\abs{d_{1}}^{4}\leq
        1+5\abs{d_{1}}+8\abs{d_{1}}^{2}+5\abs{d_{1}}^{3}$. Coupling these results
        with $\abs{d_{1}}\geq1$ gives that $1\leq \abs{d_{1}}\leq \ps$, where
        $\ps$ is a root of $x^{4}-5x^{3}-8x^{2}-5x-1$, and is approximately given
        by $6.38048$. Thus $-7<d_{1}<7$. We also find that 
        \begin{equation}
          \label{eq: RSW Inequality}
          r^{2}+s^{2}\leq\(n^{2}+4\)\frac{4\abs{d_{1}}^{3}+5\abs{d_{1}}^{2}+4\abs{d_{1}}+1}{\abs{d_{1}}^{2}\(1+\abs{d_{1}}^{2}\)}.
        \end{equation}

        Given a bound on $d_{1}$ we now have a bound on a sum of squares of
        integers and hence we can exhaust all possibilities. To do this we proceed
        in two subcases:

        \textbf{Case 2.1.1:} $n>0$\ \\
          The fact that $d_{1}=\frac{1}{2}\(n+\sqrt{n^{2}+4}\)$ implies $1\leq
          n\leq 6$ and we have the case considered in \cite{RSW}. 

          In particular, we may apply inequality (\ref{eq: RSW Inequality}),
          these bounds for $n$, and our formula for $d_{1}$ to produce a list of
          triples $\(n,r,s\)$. Just as in \cite{RSW} we may enforce integrality
          of $d_{2}d_{3}/d_{1}$, $d_{3}/d_{2}-d_{2}/d_{3}$,
          $\tilde{s}_{22}/d_{2}+\tilde{s}_{23}/d_{3}$,
          $\tilde{s}_{23}/d_{2}-\tilde{s}_{22}/d_{3}$, and
          $\tilde{s}_{22}\tilde{s}_{23}/\(d_{2}d_{3}\)$. This leads to 24 possible
          triples $\(n,r,s\)$. The Verlinde formula provides enough integrality
          conditions to further reduce these 24 triples to $8$. Of these $8$,
          only $\(n,r,s\)=\(1,-2,-1\)$ or $\(1,2,1\)$ are compatible with the
          balancing equation and the twists being roots of unity.
          In these cases one finds, $d_{1}=\t$, $d_{3}=\pm\t$ and $d_{2}=\pm1$. 
          However, these lead to relabelings of the $S$--matrices from case 1.

        \textbf{Case 2.1.2:} $n<0$\ \\
          Proceeding as in case 2.1.1, we find, by computer search, that there are 446 possible triples
          $\(n,r,s\)$ of which only 24 pass the integrality tests of \cite{RSW}.
          Applying the Verlinde formula to determine the fusion rules in these
          cases, we find that all of these either violate the integrality or
          non-negativity of the fusion coefficients.

      \textbf{Case 2.2:} $\abs{d_{1}}<1$\ \\
        Applying our Galois element, we see that $\s\(d_{1}\)=-\frac{1}{d_{1}}$.
        Setting $\d_{a}=\s\(d_{a}\)$, we find a category $\hat{\mcC}$, which is
        Galois conjugate to $\mcC$; whence if $\hat{\mcC}$ does not exist, then
        neither does $\mcC$. However, $\abs{\d_{1}}>1$ and, since Galois
        conjugation preserves all categorical identities used in case 2.1, we see
        that we must have $\d_{3}=\d_{2}\d_{1}$, $\d_{2}=\pm1$ and $\d_{1}=\t$.
        However, this is the same conclusion as in case 2.1.1. Ergo, $\mcC$ must
        be Galois conjugate to one of the case 2.1.1 results. Since these were
        conjugate to the categories determined in \cite{RSW}, we can conclude that
        $\mcC$ has an $S$--matrix Galois conjugate to one appearing in case 1. 
  \end{proof}

  Having dispensed with the symmetric and modular cases, we find that it is
  useful to stratify the properly premodular categories by self-duality and symmetric
  subcategory. It is known that that every properly premodular category has a
  symmetric subcategory \cite{M4}. Since the rank has been fixed the
  possible symmetric subcategories can be completely determined.
  \begin{proposition}
    \label{FourCases}
    \label{SymmetricSubcategory}
    If $\mcC$ is a rank 4 non-pointed properly premodular category, then there are four cases:
    \begin{enumerate}
      \item $\mcC$ is self-dual and has a symmetric subcategory Grothendieck equivalent
        to $\Rep\(\mfS_{3}\)$. 
      \item $X_{1}^{*}=X_{2}$ and generate a symmetric subcategory of $\mcC$
        Grothendieck equivalent to $\Rep\(\mbbZ/3\mbbZ\)$. 
      \item $\mcC$ is self-dual and has a symmetric subcategory Grothendieck equivalent
        to $\Rep\(\mbbZ/2\mbbZ\)$. 
      \item $\mbbI$ and $X_{1}$ generate a symmetric subcategory of $\mcC$
        Grothendieck equivalent to $\Rep\(\mbbZ/2\mbbZ\)$. Moreover, $X_{2}^{*}=X_{3}$.
    \end{enumerate}

    In each case, the symmetric subcategory is the M\"{u}ger center.
  \end{proposition}
  \begin{proof}
    We know from \cite{M4} Corollary 2.16 and comments in the
    introduction,\footnote{$\mcC'=\mcZ_{2}\(\mcC\)$ is a canonical full symmetric
    subcategory of $\mcC$.} that since $\mcC$ is nonsymmetric and nonmodular, then
    it must have a nontrivial symmetric subcategory of rank 2 or 3. Rank 3
    symmetric subcategories are known to be Grothendieck equivalent to $\Rep\(\mbbZ/3\mbbZ\)$ or
    $\Rep\(\mfS_{3}\)$ \cite{O4}. Rank 2 proceeds similarly and leads to
    $\Rep\(\mbbZ/2\mbbZ\)$. 

    In the rank 3 case, we take $X_{0}=\mbbI$, $X_{1}$, and $X_{2}$ to be
    representatives of distinct simple isomorphism classes that generate the
    symmetric subcategory, while, in rank 2, we take $X_{0}=\mbbI$ and $X_{1}$ to be the
    representative generators. The result then follows immediately by standard
    representation theory. 
  \end{proof}

  Classification of the properly premodular categories now proceeds by cases.
  The categories with high rank symmetric subcategories are, perhaps not
  surprisingly, easier to deal with since more of the datum is predetermined. As
  such, we will proceed through $\Rep\(\mfS_{3}\)$ and $\Rep\(\mbbZ/3\mbbZ\)$ first
  and then discuss the $\Rep\(\mbbZ/2\mbbZ\)$ cases.
  \begin{proposition}
    There is no rank 4 non-pointed properly premodular category with $\mcC'$
    Grothendieck equivalent to $\Rep\(\mfS_{3}\)$. 
  \end{proposition}
  \begin{proof}
    Applying the known representation theory of $\mfS_{3}$, equation
    \eqref{FusionSymmetries} and dimension counts, we find
    \begin{equation*}
      N_{1}=\(\begin{smallmatrix}
      0 & 1 & 0 & 0\\
      1 & 0 & 0 & 0\\
      0 & 0 & 1 & 0\\
      0 & 0 & 0 & 1
      \end{smallmatrix}\)
      \quad
      N_{2}=\(\begin{smallmatrix}
      0 & 0 & 1 & 0\\
      0 & 0 & 1 & 0\\
      1 & 1 & 1 & 0\\
      0 & 0 & 0 & 2
      \end{smallmatrix}\)
      \quad
      N_{3}=\(\begin{smallmatrix}
      0 & 0 & 0 & 1\\
      0 & 0 & 0 & 1\\
      0 & 0 & 0 & 2\\
      1 & 1 & 2 & M
      \end{smallmatrix}\).
    \end{equation*}

    Recall that $\tilde{s}_{ab}=d_{a}d_{b}$ for $0\leq a,b \leq2$ by
    \cite[Proposition 2.5]{M4}. Coupling this with equation
    \eqref{Balancing}, we find $\theta_{1}=\theta_{2}=1$. Denoting $\theta_{3}$
    by $\theta$, this gives 
    \begin{equation*}
      \tilde{S}=\(\begin{smallmatrix}
      1 & 1 & 2 & \frac{M\pm\sqrt{24+M^{2}}}{2}\\
      1 & 1 & 2 & \frac{M\pm\sqrt{24+M^{2}}}{2}\\
      2 & 2 & 4 & M\pm\sqrt{24+M^{2}}\\
      \frac{M\pm\sqrt{24+M^{2}}}{2} & \frac{M\pm\sqrt{24+M^{2}}}{2} & M\pm\sqrt{24+M^{2}} & \frac{12+\(M\pm\sqrt{24+M^{2}}\)M\theta}{2\theta^{2}}
      \end{smallmatrix}\).
    \end{equation*}

    Since $\frac{\tilde{s}_{33}}{\tilde{s}_{03}}$ must satisfy the
    characteristic polynomial of $N_{3}$, we can deduce that $\th$ must be a
    primitive root of unity satisfying a degree integral 3 polynomial. Employing
    the inverse of Euler's totient function, we find that $\th=\pm1$ and $M=0$. Thus
    $d=\pm\sqrt{6}$. Having removed the free parameters from this datum, we are
    in a position to prove that such a category cannot exist. In this case the M\"{u}ger
    center, $\Rep\(\mfS_{3}\)$, constitutes a Tannakian
    subcategory of $\mcC$. By
    \cite{NNW1} and \cite[Remark 5.10]{DGNO1}, we can form the
    de-equivariantization, $\mcC_{\mfS_{3}}$, which is a braided
    $\mfS_{3}$-crossed fusion category. However,
    $\FPdim\left(\mcC_{\mfS_{3}}\right)=\frac{1}{6}\FPdim\left(\mcC\right)$,
    $\dim\left(\mcC_{\mfS_{3}}\right)=\frac{1}{6}\dim\left(\mcC\right)=2$, and
    $\FPdim\left(\mcC_{\mfS_{3}}\right)=2$ \cite{DGNO1}. Thus $\mcC_{\mfS_{3}}$ is
    weakly integral braided $\mfS_{3}$-crossed fusion category and we may apply
    \cite[Corollary 8.30]{ENO1} to deduce that $\mcC_{\mfS_{3}}$ is equivalent to
    $\Rep\(\mbbZ/2\mbbZ\)$ and hence pointed. Consequently, $\mcC$ is
    group-theoretical and in particular integral, contradicting $d=\pm\sqrt{6}$
    \cite{NNW1}\cite{DGNO1}. 
  \end{proof}

  \begin{proposition}
    \label{RepZ3}
    If $\mcC$ is a non-pointed properly premodular category such that $\langle
    X_{0},X_{1},X_{2}\rangle=\mcC'$ is Grothendieck equivalent to $\Rep\(\mbbZ/3\mbbZ\)$, then:
    \begin{align*}
      \tilde{S}=\(\begin{smallmatrix}
      1 & 1 & 1 & 3\\
      1 & 1 & 1 & 3\\
      1 & 1 & 1 & 3\\
      3 & 3 & 3 & -3
      \end{smallmatrix}\)
      \quad
      T = \(\begin{smallmatrix}
      1 & 0 & 0 & 0\\
      0 & 1 & 0 & 0\\
      0 & 0 & 1 & 0\\
      0 & 0 & 0 & -1
      \end{smallmatrix}\)\ \\
      N_{1}=\(\begin{smallmatrix}
      0 & 1 & 0 & 0\\
      0 & 0 & 1 & 0\\
      1 & 0 & 0 & 0\\
      0 & 0 & 0 & 1
      \end{smallmatrix}\)
      \quad
      N_{2}=\(\begin{smallmatrix}
      0 & 0 & 1 & 0\\
      1 & 0 & 0 & 0\\
      0 & 1 & 0 & 0\\
      0 & 0 & 0 & 1
      \end{smallmatrix}\)
      \quad
      N_{3}=\(\begin{smallmatrix}
      0 & 0 & 0 & 1\\
      0 & 0 & 0 & 1\\
      0 & 0 & 0 & 1\\
      1 & 1 & 1 & 2
      \end{smallmatrix}\),
    \end{align*}
    and $\mcC$ is realized by $\mcC\(\mathfrak{sl}\(2\),6\)_{ad}$.
  \end{proposition}
  \begin{proof}
    Applying \propref{FourCases}, we know that $\mcC$ is self-dual and so
    applying the representation theory of $\mbbZ/3\mbbZ$ and equation
    \eqref{FusionSymmetries}, we find that the fusion matrices are determined up
    to $N_{33}^{3}$. Making use of equation \eqref{Balancing}, the fact that
    $\tilde{S}=\tilde{S}^{T}$, and the fact that in a properly premodular
    category some column of $\tilde{S}$ is a multiple of the first, one finds
    that
    \begin{equation*}
      \tilde{S}=\(\begin{smallmatrix}
      1&1&1&d_{3}\\
      1&1&1&d_{3}\\
      1&1&1&d_{3}\\
      d_{3}&d_{3}&d_{3}&\frac{3+d_{3}N_{33}^{3}\th_{3}}{\th_{3}^{2}}
      \end{smallmatrix}\)
      \quad 
      T=\(\begin{smallmatrix}
      1&0&0&0\\
      0&1&0&0\\
      0&0&1&0\\
      0&0&0&\th_{3}
      \end{smallmatrix}\).
    \end{equation*}

    By dimension count, we see that
    $d_{3}=\frac{1}{2}\(N_{33}^{3}\pm\sqrt{12+N_{33}^{3}}\)$. So it remains to
    determine $N_{33}^{3}$ and $\th_{3}$. For notational brevity, we let
    $M=N_{33}^{3}$. Applying equation \eqref{PreVerlinde} we find that 
    \begin{equation}
      \label{RepZ3 theta equation}
      \begin{split}
        (\th_{3}-1) \(18 \th_{3} \(\th_{3}^2+\th_{3}+1\)+\th_{3}^2 M^4+3 \th_{3} (\th_{3}+1) (\th_{3}+2)M^2+18\)\\
        =\pm(\th_{3}-1) \(3 \th_{3} \(\th_{3}^2+\th_{3}+2\) \sqrt{M^2+12} M+\th_{3}^2 \sqrt{M^2+12} M^3\).
      \end{split}
    \end{equation}

    We first note that if $\th_{3}=1$, then $\mcC=\mcC'$ contradicting the
    nonsymmetric assumption. Thus, $\th_{3}$ satisfies a degree 6 integral
    polynomial. However, $\th_{3}$ is a root of unity, so applying the inverse
    Euler phi function to determine a list of potential values for $\th_{3}$.
    Combing the possible cases, one finds $N_{33}^{3}\in\lcb0,2\rcb$ and
    $\th_{3}\in\lcb\pm i,-1\rcb$. Applying \corref{Modular FS is Algebraic
    Integer} with $a=3$, we find that only $N_{33}^{3}=2$ gives a rational
    integer. Evaluating equation \eqref{RepZ3 theta equation} at $N_{33}^{3}=2$
    reveals that $\th=-1$ is the only solution.\footnote{If one proceeds without
    appealing to the Frobenius-Schur indicators then the Tambara-Yamagami with
    dimensions $1,1,1,\sqrt{3}$ appear. This can of course be excluded since such
    categories do not admit a braiding \cite{S1}} 
  \end{proof}

  Having dispensed with the large symmetric subcategories, we need to consider
  the case that $\Rep\(\mbbZ/2\mbbZ\)$ appears as a symmetric subcategory. We first
  consider the non-self-dual case which can be dealt with by cyclotomic/number
  theoretic techniques.
  \begin{proposition}
    There is no rank 4 non-pointed properly premodular category such that
    $\langle X_{0},X_{1}\rangle=\mcC'$ is Grothendieck equivalent to $\Rep\(\mbbZ/2\mbbZ\)$, and
    $X_{2}^{*}=X_{3}$.
  \end{proposition}
  \begin{proof}
    Given the standard representation theory of $\mbbZ/2\mbbZ$ and the equation
    \eqref{FusionSymmetries}, we immediately obtain:
    \begin{equation*}
      N_{1}=\(\begin{smallmatrix}
      0 & 1 & 0 & 0\\
      1 & 0 & 0 & 0\\
      0 & 0 & N_{32}^{1} & N_{33}^{1}\\
      0 & 0 & N_{33}^{1} & N_{32}^{1}
      \end{smallmatrix}\)
      \quad
      N_{2}=\(\begin{smallmatrix}
      0 & 0 & 1 & 0\\
      0 & 0 & N_{32}^{1} & N_{33}^{1}\\
      0 & N_{33}^{1} & N_{33}^{3} & N_{33}^{2}\\
      1 & N_{32}^{1} & N_{33}^{3} & N_{33}^{3}
      \end{smallmatrix}\)
      \quad
      N_{3}=\(\begin{smallmatrix}
      0 & 0 & 0 & 1\\
      0 & 0 & N_{33}^{1} & N_{32}^{1}\\
      1 & N_{32}^{1} & N_{33}^{3} & N_{33}^{3}\\
      0 & N_{33}^{1} & N_{33}^{2} & N_{33}^{3}
      \end{smallmatrix}\).
    \end{equation*}
    Demanding that the fusion matrices mutually commute reveals that either
    $N_{32}^{1}$ or $N_{33}^{1}$ is 0 and the other is 1. Hence, the proof
    bifurcates into two cases.

    \textbf{Case 1:} $N_{32}^{1}=1$ and $N_{33}^{1}=0$\\
      Returning to the commutativity of the fusion matrices, we are reduced to one
      equation:
      \begin{equation*}
        2= \(N_{33}^{2}\)^{2}-\(N_{33}^{3}\)^{2}=\(N_{33}^{2}-N_{33}^{3}\)\(N_{33}^{2}+N_{33}^{3}\)
      \end{equation*}

      Of course the fusion coefficients are non-negative integers and so
      $N_{33}^{2}-N_{33}^{3}=1$ and $N_{33}^{2}+N_{33}^{3}=2$. Of course this
      system has no solution in $\mbbZ$.

    \textbf{Case 2:} $N_{32}^{1}=0$ and $N_{33}^{1}=1$.\\
      In this case the commutativity of the fusion matrices reveals that
      $N_{33}^{2}=N_{33}^{3}$, which we will simply call $M$ for brevity. Applying
      the equation \eqref{Balancing}, and dimension count, we can determine the
      $S$--matrix to be
      \begin{equation*}
        \tilde{S}=\(\begin{smallmatrix}
        1 & M\pm\sqrt{1+M^{2}}\\
        M\pm\sqrt{1+M^{2}} & \frac{1+2\(M\pm\sqrt{1+M^{2}}\)M\th}{\th^{2}}
        \end{smallmatrix}\)\otimes\(\begin{smallmatrix}
        1 & 1\\
        1 & 1
        \end{smallmatrix}\).
      \end{equation*}
      Where $\th:=\theta_{2}=\theta_{3}$ and $\th_{1}=1$, which follows from the
      fact that some column of the $S$--matrix must be a multiple of the first
      \cite{Brug1}. However, $\frac{\tilde{s}_{22}}{\tilde{s}_{02}}$ must
      satisfy the characteristic polynomial of $N_{2}$, which factors into two
      quadratics. Inserting this quotient into the factors, we find that $\th$
      must satisfy either a degree 4 or degree 8 polynomial over $\mbbZ$. Since
      $\th$ is a primitive root of unity we can apply the inverse Euler phi
      function to bound the degree of the minimal polynomial of $\th$. Proceeding
      through all cases, we find that $M=0$ and $\mcC$ is pointed. 
  \end{proof}

  While this cyclotomic analysis has been quite fruitful, the remaining,
  properly premodular case proves to be resistant and so other approaches are
  necessary. We begin by recalling that every fusion category admits a (possibly
  trivial) grading. Since the category has small rank, the grading possibilities
  allow for further stratification of the problem.
  \begin{proposition}
    If $\mcC$ is a self-dual rank 4 non-pointed properly premodular category
    $\langle X_{0},X_{1}\rangle=\mcC'$ is Grothendieck equivalent to $\Rep\(\mbbZ/2\mbbZ\)$, then there are
    three cases.
    \begin{enumerate}
      \item $\mcC$ admits a universal $\mbbZ/2\mbbZ$ grading
      \item $\mcC$ does not admit a universal $\mbbZ/2\mbbZ$ grading and $X_{1}\otimes X_{2}=X_{2}$
      \item $\mcC$ does not admit a universal $\mbbZ/2\mbbZ$ grading and $X_{1}\otimes X_{2}=X_{3}$
    \end{enumerate}
  \end{proposition}
  \begin{proof}
    If $\mcC$ admits a nontrivial universal grading, then it must be by
    $\mbbZ/2\mbbZ$.  On the other hand, if $\mcC$ does not admit a universal
    grading, then $\mcC_{ad}=\mcC$ \cite{DGNO1}. Since $X_{1}$ generates
    $\mcC'\cong\Rep\(\mbbZ/2\mbbZ\)$, we can conclude that if $\mcC_{ad}=\mcC$ then
    either $X_{1}\otimes X_{2}=X_{2}$ or $X_{1}\otimes X_{2}=X_{3}$. 
  \end{proof}

  With this proposition in hand we again proceed by cases. First, we consider
  with the relatively simple case: $\mcC$ admits a universal $\mbbZ/2\mbbZ$
  grading.
  \begin{proposition}
    \label{Z2GradedZ2Subcat}
    Suppose $\mcC$ is a self-dual rank 4 non-pointed properly premodular
    category admitting a universal $\mbbZ_{2}$ grading such that $\mcC'$ is
    Grothendieck equivalent to $\Rep\(\mbbZ_{2}\)$, then $\mcC$ is a Deligne product of the $\Fib$ with
    $\Rep\(\mbbZ/2\mbbZ\)$ or $\sVec$.
  \end{proposition}
  \begin{proof}
    Dimension count coupled with the representation theory of $\mbbZ/2\mbbZ$
    completely determines the fusion relations up to $N_{22}^{2}$. However, we
    can apply \cite{O4} to conclude that $N_{22}^{2}\in\lcb0,1\rcb$.
    $N_{22}^{2}=0$ leaves a pointed category and so we must have $N_{22}^{2}=1$,
    and $d:=d_{2}=d_{3}=\frac{1\pm\sqrt{5}}{2}$. Applying equation
    \eqref{Balancing} and the fact that a column of the $S$--matrix must be a
    multiple of the first we find that $\th_{1}=\pm1$,
    $\th:=\th_{2}=\th_{1}\th_{3}$, and
    \begin{equation*}
      \tilde{S}=\(\begin{smallmatrix}
      1 & 1 & d & d\\
      1 & 1 & d & d\\
      d & d & \frac{1+d\th}{\th^{2}} &  \frac{1+d\th}{\th^{2}}\\
      d & d & \frac{1+d\th}{\th^{2}} & \frac{1+d\th}{\th^{2}}
      \end{smallmatrix}\)
      \quad
      T=\(\begin{smallmatrix}
      1 & 0 & 0 & 0\\
      0 & \pm1 & 0 & 0\\
      0 & 0 & \th & 0\\
      0 & 0 & 0 & \pm\th
      \end{smallmatrix}\).
    \end{equation*}

    Since the normalized columns of the $S$--matrix are characters of the fusion
    ring, it must be that $\frac{1+d\th}{d\th^{2}}$ is a simultaneous root of
    the characteristic polynomials of $N_{2}$ and $N_{3}$. This gives the
    desired result.
  \end{proof}

  Finally, we come to the last two cases where $\mcC'$ is Grothendieck
  equivalent to $\Rep\(\mbbZ_{2}\)$ and the universal grading group is trivial. 
  These are by far the most complicated cases. To
  dispense with the first case we make use of the minimal modularization
  \cite{Brug1}.
  \begin{proposition}
    Suppose $\mcC$ is a self-dual, rank 4, non-pointed, properly premodular
    category such that $\mcC'$ is Grothendieck equivalent to $\Rep\(\mbbZ/2\mbbZ\)$, $\mcC$ does not admit a
    nontrivial universal grading, and $X_{1}\otimes X_{2} = X_{2}$, then 
    \begin{align*}
    \tilde{S}=\(\begin{smallmatrix}
    1 & 1 & 2 & 2\\
    1 & 1 & 2 & 2\\
    2 & 2 & \frac{2+2\theta}{\theta^{3}} & \frac{2+2\theta^{2}}{\theta}\\
    2 & 2 & \frac{2+2\theta^{2}}{\theta} & \frac{2+2\theta}{\theta^{3}} 
    \end{smallmatrix}\)
    \quad
    T=\(\begin{smallmatrix}
    1 & 0 & 0 & 0\\
    0 & 1 & 0 & 0\\
    0 & 0 & \theta & 0\\
    0 & 0 & 0 & \theta^{-1}
    \end{smallmatrix}\)
    \ \\
    N_{1}=\(\begin{smallmatrix}
    0 & 1 & 0 & 0\\
    1 & 0 & 0 & 0\\
    0 & 0 & 1 & 0\\
    0 & 0 & 0 & 1
    \end{smallmatrix}\)
    \quad
    N_{2} =\(\begin{smallmatrix}
    0 & 0 & 1 & 0\\
    0 & 0 & 1 & 0\\
    1 & 1 & 0 & 1\\
    0 & 0 & 1 & 1
    \end{smallmatrix}\)
    \quad
    N_{3}=\(\begin{smallmatrix}
    0 & 0 & 0 & 1\\
    0 & 0 & 0 & 1\\
    0 & 0 & 1 & 1\\
    1 & 1 & 1 & 0
    \end{smallmatrix}\),
    \end{align*}

    and $\th$ is a primitive $5^{\text{th}}$ root of unity. Such datum are
    realized by $\mcC=\mcC\(\mathfrak{so}\(5\),10\)_{ad}$.
  \end{proposition}
  \begin{proof}
    The representation theory of $\mbbZ/2\mbbZ$, dimension count, equation
    \eqref{FusionSymmetries}, and equation \eqref{Balancing} give
    \begin{align*}
      N_{1}=\(\begin{smallmatrix}
      0 & 1 & 0 & 0\\
      1 & 0 & 0 & 0\\
      0 & 0 & 1 & 0\\
      0 & 0 & 0 & 1
      \end{smallmatrix}\)
      \quad
      N_{2}=\(\begin{smallmatrix}
      0 & 0 & 1 & 0\\
      0 & 0 & 1 & 0\\
      1 & 1 & N_{22}^{2} & N_{32}^{2}\\
      0 & 0 & N_{32}^{2} & N_{33}^{2}
      \end{smallmatrix}\)
      \quad
      N_{3}=\(\begin{smallmatrix}
      0 & 0 & 0 & 1\\
      0 & 0 & 0 & 1\\
      0 & 0 & N_{32}^{2} & N_{33}^{2}\\
      1 & 1 & N_{33}^{2} & N_{33}^{3}
      \end{smallmatrix}\)\\
      \tilde{S}=\(\begin{smallmatrix}
      1 & 1 & d_{2} & d_{3}\\
      1 & 1 & d_{2} & d_{3}\\
      d_{2} & d_{2} & \frac{2+d_{2}N_{22}^{2}\th_{2}+d_{3}N_{32}^{2}\th_{3}}{\th_{2}^{2}} & \frac{d_{2}N_{32}^{2}\th_{2}+d_{3}N_{33}^{2}\th_{3}}{\th_{2}\th_{3}}\\
      d_{3} & d_{3} & \frac{d_{2}N_{32}^{2}\th_{2}+d_{3}N_{33}^{2}\th_{3}}{\th_{2}\th_{3}} & \frac{2+d_{2}N_{33}^{2}\th_{2}+d_{3}N_{33}^{3}\th_{3}}{\th_{3}^{2}}
      \end{smallmatrix}\)
      \quad
      T=\(\begin{smallmatrix}
      1 & 0 & 0 & 0\\
      0 & 1 & 0 & 0\\
      0 & 0 & \th_{2} & 0\\
      0 & 0 & 0 & \th_{3}
      \end{smallmatrix}\).
    \end{align*}

    Applying \cite[Proposition 4.2]{Brug1}, we can deduce that $\mcC$
    admits a modularization $\hat{\mcC}$. We can now apply \cite{Brug1}
    Proposition 4.4 and the equivalence between Brugui\`{e}res modularization
    and the de-equivariantization to deduce that $\hat{\mcC}$ is a rank 5
    modular category with simple isomorphism classes of simple objects 
    $\mbbI, Y_{1}, Y_{2}, Z_{1}, Z_{2}$ such that
    $Y_{i}^{*}\in\lcb Y_{1}, Y_{2}\rcb$ and $Z_{i}^{*}\in\lcb Z_{1}, Z_{2}\rcb$.
    Applying the classification of rank 5 modular categories, \cite{BNRW2}, we
    can conclude that $\hat{\mcC}$ is pointed and hence $d_{2}=\pm2$ and $d_{3}=\pm2$. 
    Dimension count then allows us to eliminate all fusion coefficients except for $N_{33}^{3}$. Applying
    \cite[Theorem 4.2]{NR1}, we know that $\mcC$ is Grothendieck equivalent to
    $\Rep\(D_{n}\)$ and is group-theoretical. This gives $d_{2}=d_{3}=2$, and
    determines the fusion coefficients. Applying equations \eqref{PreVerlinde}
    and \eqref{STCubed and STInvCubed}, we find $\th_{3}=\th_{2}^{-1}$ and that
    $\th_{2}$ is a primitive 5th root of unity. 
  \end{proof}

  The final case requires not only the minimal modularization of Brugui\`{e}res
  but also the second Frobenius-Schur indicators.
  \begin{proposition}
    Suppose $\mcC$ is a self-dual rank 4 non-pointed properly premodular
    category such that $\mcC'$ is Grothendieck equivalent to $\Rep\(\mbbZ/2\mbbZ\)$, 
    $\mcC$ does not admit a nontrivial universal grading, and $X_{1}\otimes X_{2}=X_{3}$ then
    \begin{align*}
      \tilde{S}&=\(\begin{smallmatrix}
      1 & 1\pm\sqrt{2}\\
      1\pm\sqrt{2} & -1
      \end{smallmatrix}\)
      \otimes
      \(\begin{smallmatrix}
      1 & 1 \\ 1 & 1
      \end{smallmatrix}\)
      \quad
      T = \(\begin{smallmatrix}
      1 & 0 & 0 & 0\\
      0 & -1 & 0 & 0\\
      0 & 0 & i& 0 \\
      0 & 0 & 0 & -i
      \end{smallmatrix}\)
      \ \\
      N_{2}&=\(\begin{smallmatrix}
      0 & 0 & 1 & 0\\
      0 & 0 & 0 & 1\\
      1 & 0 & 1 & 1\\
      0 & 1 & 1 & 1
      \end{smallmatrix}\)
      \quad
      N_{3}=\(\begin{smallmatrix}
      0 & 0 & 0 & 1\\
      0 & 0 & 1 & 0\\
      0 & 1 & 1 & 1\\
      1 & 0 & 1 & 1
      \end{smallmatrix}\),
    \end{align*}
    such a datum is realized by $\mcC=\mcC\(\mathfrak{sl}\(2\),8\)_{ad}$ and its conjugates.
  \end{proposition}
  \begin{proof}
    Applying dimension count, equation \eqref{FusionSymmetries}, and the usual
    representation theory for $\mbbZ/2\mbbZ$, we can determine the fusion rules up
    to two parameters:
    \begin{equation*}
      N_{1}=\(\begin{smallmatrix}
      0 & 1 & 0 & 0\\
      1 & 0 & 0 & 0\\
      0 & 0 & 0 & 1\\
      0 & 0 & 1 & 0
      \end{smallmatrix}\)
      \quad
      N_{2}=\(\begin{smallmatrix}
      0 & 0 & 1 & 0\\
      0 & 0 & 0 & 1\\
      1 & 0 & N & M\\
      0 & 1 & M & N
      \end{smallmatrix}\)
      \quad
      N_{3}=\(\begin{smallmatrix}
      0 & 0 & 0 & 1\\
      0 & 0 & 1 & 0\\
      0 & 1 & M & N\\
      1 & 0 & N & M
      \end{smallmatrix}\).
    \end{equation*}

    Furthermore, we can deduce that $M,N\neq0$ lest we reduce to the fusion
    rules of \propref{Z2GradedZ2Subcat} or a pointed category. Next, we may use
    equation \eqref{Balancing}, dimension count, and that
    $\tilde{s}_{ij}=\l\tilde{s_{i0}}$ for some $j$ and some
    $\l\in\mbbC^{\times}$, to find the $S$-- and $T$--matrices:
    \begin{align*}
      \tilde{S}&=\(\begin{smallmatrix}
      1 & \frac{N+M+\e\sqrt{4+\(M+N\)^{2}}}{2}\\
      \frac{N+M+\e\sqrt{4+\(M+N\)^{2}}}{2} & \frac{2+\(N\theta+\delta M\theta\)\(N+M+\e\sqrt{4+\(M+N\)^{2}}\)}{2\theta^{2}}
      \end{smallmatrix}\)\otimes\(\begin{smallmatrix}
      1 & 1\\
      1 & 1
      \end{smallmatrix}\)
      \ \\
      T&=\(\begin{smallmatrix}
      1 & 0 & 0 & 0\\
      0 & \d & 0 & 0\\
      0 & 0 & \th & 0\\
      0 & 0 & 0  & \d\th
      \end{smallmatrix}\),
    \end{align*}

    where $\e,\d=\pm1$. We treat $\d=1$ and $\d=-1$ in separate cases.

    \textbf{Case 1:} $\d=1$\ \\
      Here we can apply \cite{Brug1} to deduce that $\mcC$ is
      modularizable. Letting $H:\mcC\to\hat{\mcC}$ denote its minimal
      modularization then we have $X_{2}\in M_{\mcC}X_{3}$ and so
      $H\(X_{2}\)\cong H\(X_{3}\)$. Furthermore, $\|Stab_{M_{\mcC}}X\|=1$ for
      all simple $X$ and thus, $\dim H\(X_{2}\)=\dim\(X_{2}\)$. Consequently,
      the trivial object in $\hat{\mcC}$ as well as $H\(X_{2}\)$ account for
      $1+d^{2}$ of the dimension of $\hat{\mcC}$. However,
      $\dim\hat{\mcC}=\frac{1}{2}{\dim\(\mcC\)}=1+d^{2}$ and so $\hat{\mcC}$ is
      a rank 2 modular category with simple objects $\mbbI$ and $H\(X_{2}\)$.
      Such categories have been classified in \cite{RSW} and are the Semion and
      the Fibonacci. In these situations, we find either that $\mcC$ is pointed
      or that $M=N=0$ and so we can exclude the case of $\d=1$.

    \textbf{Case 2:} $\d=-1$.\ \\
      A straightforward application of \eqref{PreVerlinde} and \eqref{STCubed
      and STInvCubed} in a computer algebra system is used to further reduces the solution space. Discarding any
      solutions where either $M$ or $N$ is 0 or $\mcC$ is symmetric leaves 7
      possible families of solutions. One of these families contains a
      Pythagorean triple with 1 which forces $N<0$ and hence can be discarded.
      Two of the other families of solutions have $M$ and $N$ related by
      \begin{equation*}
        M=\frac{-N\th^{2}\pm\sqrt{-\th\(1+\th^{2}\)^{2}\(1-\(1+N^{2}\)\th+\th^{2}\)}}{\th\(1+\th\(\th-1\)\)}.
      \end{equation*}

      Since $\th\neq0$, this can be arranged into a monic integral degree 6
      polynomial $\th$. Since $\th$ is a root of unity we can apply the inverse
      Euler phi function to find a possible list of values for $\th$. Direct
      calculation reveals that none of these roots of unity can satisfy this
      polynomial in a manner consistent with $M, N> 0$.

    The remaining four families can be reduced by resolving a labeling ambiguity
    to give
    \begin{align*}
      \tilde{S}&=\(\begin{smallmatrix}
      1 & N+\e\sqrt{1+N^{2}}\\
      N+\e\sqrt{1+N^{2}} & -1
      \end{smallmatrix}\)
      \otimes
      \(\begin{smallmatrix}
      1 & 1 \\ 1 & 1
      \end{smallmatrix}\)
      \quad
      T = \(\begin{smallmatrix}
      1 & 0 & 0 & 0\\
      0 & -1 & 0 & 0\\
      0 & 0 & i& 0 \\
      0 & 0 & 0 & -i
      \end{smallmatrix}\)\ \\
      N_{2}&=\(\begin{smallmatrix}
      0 & 0 & 1 & 0\\
      0 & 0 & 0 & 1\\
      1 & 0 & N & N\\
      0 & 1 & N & N
      \end{smallmatrix}\)
      \quad
      N_{3}=\(\begin{smallmatrix}
      0 & 0 & 0 & 1\\
      0 & 0 & 1 & 0\\
      0 & 1 & N & N\\
      1 & 0 & N & N
      \end{smallmatrix}\).
    \end{align*}

    Applying \corref{Modular FS is Algebraic Integer} to $X_{2}$, we find that
    $N\pm\frac{N^{2}-1}{\sqrt{N^{2}+1}}\in\mbbZ$. Denoting this integer by $L$
    and simplifying we find 
    \begin{align*}
      4=\(N^{2}+1\)\(3+L^{2}-2 L N\)
    \end{align*}

    However, $N^{2}+1\neq0$ and so, reducing modulo $N^{2}+1$, we find that
    $4\equiv 0\mod N^{2}+1$.

    This only occurs for $N\in\lcb-1,0,1\rcb$. Since $N=0$ leads to $\mcC$ being
    pointed and we know $N\geq0$, we can conclude that $N=1$.
  \end{proof}

  The results of this section can be compiled to give the following theorem.
  \begin{theorem}
    \label{Rank4Premodular}
    If $\mcC$ is a non-pointed rank 4 premodular category, then exactly one of the following is true
    \begin{enumerate}
      \item $\mcC$ is symmetric and is Grothendieck equivalent to
        $\Rep\(G\)$ where $G$ is $\mbbZ/4\mbbZ$, $\mbbZ/2\mbbZ\times\mbbZ/2\mbbZ$,
        $D_{10}$, or $\mfA_{4}$.
      \item $\mcC$ is properly premodular and is Grothendieck equivalent to
        a Galois conjugate of one of the following: $\mcC\(sl\(2\),8\)_{ad}$,
        $\mcC\(sl\(2\),6\)_{ad}$, $\mcC\(so\(5\),10\)_{ad}$,
        $\Fib\boxtimes\Rep\(\mbbZ/2\mbbZ\)$, or $\Fib\boxtimes \sVec$.
      \item $\mcC$ is modular and is Galois conjugate to a modular category
        from \cite{RSW} or has $S$-matrix
        \begin{equation*}
          \(\begin{smallmatrix}
            1&-1&\bar{\t}&\t\\
            -1&1&-\t&-\bar{\t}\\
            \bar{\t}&-\t&-1&-1\\
            \t&-\bar{\t}&-1&-1
          \end{smallmatrix}\),
        \end{equation*}

        where $\t=\frac{1+\sqrt{5}}{2}$ is the golden mean and
        $\bar{\t}=\frac{1-\sqrt{5}}{2}$ is its Galois conjugate.
    \end{enumerate}
  \end{theorem}

\section*{Acknowledgments}
  I'd like to thank Eric Rowell for very useful discussions, guidance, and
  inspiration. I'd also like to thank C\/{e}sar Galindo, Tobias Hagge, David
  Jordan, Siu-Hung Ng, Julia Plavnik, Matthew Titsworth, and Zhenghan Wang for 
  enlightening discussions.

\clearpage
\pagestyle{plain}

\bibliography{References}
\bibliographystyle{plain}

\end{document}